\documentclass[11pt]{article}

\usepackage{graphicx}

\usepackage{bbm}

\usepackage{amsmath}
\usepackage{amsthm}
\usepackage{amsfonts}
\usepackage{amssymb}
\usepackage{xcolor}

\newcommand{\eee}{{\rm e}}

\newcommand{\me}{\mathbb{E}}
\newcommand{\mn}{\mathbb{N}}
\newcommand{\mr}{\mathbb{R}}

\DeclareMathOperator{\1}{\mathbbm{1}}

\newcommand{\mmp}{\mathbb{P}}

\newtheorem{thm}{Theorem}[section]
\newtheorem{lemma}[thm]{Lemma}

\newtheorem{assertion}[thm]{Proposition}
\theoremstyle{definition}

\theoremstyle{remark}
\newtheorem{rem}[thm]{Remark}

\begin{document}
\title{Precise tail behavior of some Dirichlet series}\date{}
\author{Alexander Iksanov\footnote{Faculty of Computer Science and Cybernetics, Taras Shevchenko National University of Kyiv, Ukraine; e-mail address:
iksan@univ.kiev.ua}
\ \ and \ \ Vitali Wachtel\footnote{Faculty of Mathematics, University of Bielefeld, Germany; e-mail address: wachtel@math.uni-bielefeld.de}}
\maketitle

\begin{abstract}
Let $\eta_1$, $\eta_2,\ldots$ be independent copies of a random variable $\eta$ with zero mean and finite variance which is 
bounded from the right, that is, $\eta\leq b$ almost surely for some $b>0$. Considering different types of the asymptotic behaviour of the probability $\mathbb{P}\{\eta\in[b-x,b]\}$ as $x\to 0+$ we derive precise tail asymptotics of the random Dirichlet series $\sum_{k\geq 1}k^{-\alpha}\eta_k$ for $\alpha\in (1/2, 1]$.
\end{abstract}


\hphantom{2020 Mathematics Subject Classification: } Secondary: 60F05

\section{Introduction and main results}
Let $\eta_1$, $\eta_2,\ldots$ be independent copies of a random variable $\eta$ with zero mean and finite variance. By Kolmogorov's three-series theorem, the random Dirichlet series
$$
S(\alpha):=\sum_{k\geq 1}k^{-\alpha}\eta_k
$$
converges almost surely (a.s.) if, and only if, $\alpha>1/2$. Throughout the paper we additionally assume $\eta\leq b$ a.s.\ for some $b>0$. Our purpose is to find a precise (as opposed to logarithmic) asymptotic behavior of the distribution right tail of $S(\alpha)$, that is, $\mmp\{S(\alpha)>x\}$ as $x\to\infty$. If $\alpha>1$, then $S(\alpha)\leq b\zeta(\alpha)$ a.s., where $\zeta$ is the Riemann zeta function. Thus, in what follows our attention is restricted to the non-trivial case $\alpha\in (1/2, 1]$. At the end of Section \ref{sect:discussion} we discuss this case and the assumption $\eta\le b$ in more details. The other part of Section \ref{sect:discussion} is concerned with the complementary situation, in which the random variable $\eta$ is a.s.\ unbounded. 

Put $\psi(t):=\log \me [ \eee^{t\eta}]$ for $t\in\mr$. Under the present assumptions, the function $\psi$ is infinitely differentiable, strictly convex and strictly increasing on $[0,\infty)$. In particular, $\psi^\prime$ is positive and strictly increasing on $(0,\infty)$ and $\psi^{\prime\prime}$ is positive on $[0,\infty)$. Recall that, for $\rho\in (0,1]$, the Euler-Mascheroni constant $\gamma_\rho$ is defined by
\begin{equation}\label{eq:EulMas<1}
\gamma_\rho:=\lim_{n\to\infty}\Big(\sum_{k=1}^n k^{-\rho}-(1-\rho)^{-1}n^{1-\rho}\Big)
\end{equation}
if $\rho\in (0,1)$, and
\begin{equation}\label{eq:EulMas=1}
\gamma_1:=\lim_{n\to\infty}\Big(\sum_{k=1}^n k^{-1}-\log n\Big)
\end{equation}
if $\rho=1$.
\begin{thm}\label{thm:main1}
Assume that $\eta\leq b$ a.s.\ with $\mmp\{\eta=b\}=\theta\in (0,1)$, $\me[\eta]=0$ and $\me [\eta^2]<\infty$.

\noindent If $\alpha=1$, then, as $x\to\infty$,
$$
\mmp\{S(1)>x\}~\sim~\frac{1}{(2\pi\theta b)^{1/2}}\exp\Big(-\frac{x-q}{2b}-b\eee^{(x-q) 
/b}\Big),$$ where $$q=b\gamma_1+\int_0^1 x^{-1}\psi^\prime(x){\rm d}x+\int_1^\infty x^{-1}(\psi^\prime(x)-b){\rm d}x\in\mr.
$$

\noindent
If $\alpha\in (1/2, 1)$, then, as $x\to\infty$,
\begin{multline*}
\mmp\{S(\alpha)>x\}\\~\sim~\frac{1}{(2\pi\theta)^{1/2}}\Big(\frac{\alpha(\sigma_\alpha^2)^\alpha}{1-\alpha}\Big)^{1/(2(1-\alpha))}\frac{1}{x^{1/(2(1-\alpha))}}
\exp\Big(-\Big(\frac{1-\alpha}{(\alpha\sigma_\alpha^2)^\alpha}\Big)^{1/(1-\alpha)}(x-b\gamma_\alpha)^{1/(1-\alpha)}\Big),
\end{multline*}
where
$$
\sigma_\alpha^2:=\alpha^{-1}\int_0^\infty x^{1-1/\alpha}\psi^{\prime\prime}(x){\rm d}x=(1-\alpha)\alpha^{-3}\int_0^\infty x^{-1-1/\alpha}\psi(x){\rm d}x\in (0,\infty).
$$
\end{thm}

\begin{thm}\label{thm:main2}
Assume that $\eta\leq b$ a.s., $\mmp\{b-\eta\leq x\}\sim \lambda x^r$ as $x\to 0+$ for positive $\lambda$ and $r$,  $\me[\eta]=0$ and $\me [\eta^2]<\infty$.

\noindent If $\alpha=1$, then, as $x\to\infty$, $$\mmp\{S(1)>x\}~\sim~\Big(\frac{(2\pi)^{r-1}}{\lambda\Gamma(r+1)b}\Big)^{1/2}\exp\Big(\frac{r-1}{2b}(x-q)-b\eee^{(x-q) 
/b}\Big)$$ with $q$ as defined in Theorem \ref{thm:main1}.

\noindent If $\alpha\in (1/2, 1)$, then, as $x\to\infty$,
\begin{multline*}
\mmp\{S(\alpha)>x\}~\sim~\Big(\frac{(2\pi)^{r\alpha-1}}{\lambda\Gamma(r+1)}\Big)^{1/2}\Big(\frac{(1-\alpha)^{r\alpha-1}}{\alpha^{r\alpha-1}(\sigma_\alpha^2)^{(r-1)\alpha}}\Big)^{1/(2(1-\alpha))}\exp\Big(\frac{r}{2}
\Big(\Big(\frac{\alpha\sigma_\alpha^2}{1-\alpha}\Big)^{\alpha/(1-\alpha)}-1\Big)\Big)\\
\times x^{(r\alpha-1)/(2(1-\alpha))}\exp\Big(-\Big(\frac{1-\alpha}{(\alpha\sigma_\alpha^2)^\alpha}\Big)^{1/(1-\alpha)}(x-b\gamma_\alpha)^{1/(1-\alpha)}\Big)
\end{multline*}
with 
$\sigma^2_\alpha$ as defined in Theorem \ref{thm:main1}.
\end{thm}
\begin{rem}
Putting formally $r=0$ and $\lambda=\theta$ 
on the right-hand sides of the tail asymptotics of Theorem \ref{thm:main2} we obtain the right-hand sides of the tail asymptotics of Theorem \ref{thm:main1}, as it must be. Thus, in principle the two theorems could have been combined into a single result. The same remark also concerns Theorems \ref{thm:dens1} and \ref{thm:dens2} stated below. 
\end{rem}

The random variable $S(\alpha)$ can be seen as a special case of the series $\sum_{k\ge1}c_k\eta_k$, where $c_1$, $c_2,\ldots$ are real numbers. Perhaps, the most known representative of this family is called the Bernoulli convolution. It corresponds to $c_k=a^k$ with some $a\in(0,1)$ and $\eta$ having a Rademacher distribution. 
The main question for the Bernoulli convolutions is whether the distribution of the series $\sum_{k\ge 1}a^k\eta_k$ is absolutely continuous with respect to (w.r.t.) Lebesgue measure or not. It has long been known 
and is quite easy to prove 
that the distribution of $\sum_{k\ge 1}a^k\eta_k$ is singular continuous for every $a<1/2$. On the other hand, 
if $a=1/2$, then
$\sum_{k\ge 1}a^k\eta_k$ has a uniform distribution on $[-1,1]$. Solomyak~\cite{Solomyak:1995} has shown that the distribution of $\sum_{k\ge 1}a^k\eta_k$ is absolutely continuous for almost all, w.r.t.\ 
Lebesgue measure, $a\in(1/2,1)$.

It has been noticed by Yaskov~\cite{Yaskov:2011} that the results of Reich~\cite{Reich:1986} can be used to show 
that the distribution of $S(\alpha)$ is absolutely continuous whenever a random variable 
$\eta$ has zero mean and finite variance. Furthermore, there exists a smooth version of the density that we denote by 
$g_\alpha$. 
Assuming that $\eta$ has a Rademacher distribution Yaskov~\cite{Yaskov:2011} has found an asymptotic behaviour of $\log g_\alpha(x)$ as $x\to\infty$. 
Our approach 
allows us to determine a precise, rather than logarithmic, 
asymptotic of $g_\alpha$.
\begin{thm}\label{thm:dens1}
Assume that $\eta\leq b$ a.s.\ with $\mmp\{\eta=b\}=\theta\in (0,1)$, $\me[\eta]=0$ and $\me[\eta^2]<\infty$.

\noindent If $\alpha=1$, then, as $x\to\infty$,
$$
g_1(x)~\sim~\frac{1}{(2\pi\theta b)^{1/2}}
\exp\Big(\frac{x-q}{2b} 
-b\eee^{(x-q) 
/b}\Big)
$$
with $q$ as defined in Theorem \ref{thm:main1}.

\noindent
If $\alpha\in (1/2, 1)$, then, as $x\to\infty$,
\begin{multline*}
g_\alpha(x)\\~\sim~\frac{1}{(2\pi\theta)^{1/2}}\Big(\frac{(1-\alpha)^{2\alpha-1}}{\alpha^{2\alpha-1}(\sigma_\alpha^2)^\alpha}\Big)^{1/(2(1-\alpha))} 
x^{(2\alpha-1)/(2(1-\alpha))}\exp\Big(-\Big(\frac{1-\alpha}{(\alpha\sigma_\alpha^2)^\alpha}\Big)^{1/(1-\alpha)}(x-b\gamma_\alpha)^{1/(1-\alpha)}\Big) 
\end{multline*}
with $\sigma^2_\alpha$ as defined in Theorem \ref{thm:main1}.
\end{thm}

One of the earliest works on the distribution of $S(\alpha)$ is the paper \cite{Rice:1973} by Rice, who has studied the case $\alpha=1$ and $\mmp\{\eta=\pm1\}=1/2$. He has obtained a version of Theorem~\ref{thm:dens1} for this particular case. It is worth mentioning that his interest on $S(\alpha)$ was motivated by applications to digital communication systems.

\begin{thm}\label{thm:dens2}
Assume that $\eta\leq b$ a.s., $\mmp\{b-\eta\leq x\}\sim \lambda x^r$ as $x\to 0+$ for positive $\lambda$ and $r$,  $\me[\eta]=0$ and $\me [\eta^2]<\infty$.

\noindent If $\alpha=1$, then
$$
g_1(x)~\sim~\Big(\frac{(2\pi)^{r-1}}{\lambda\Gamma(r+1)b}\Big)^{1/2}\exp\Big(\frac{r+1}{2b}(x-q)-b\eee^{(x-q) 
/b}\Big),\quad x\to\infty.
$$ with $q$ as defined in Theorem \ref{thm:main1}.

\noindent If $\alpha\in (1/2, 1)$, then, as $x\to\infty$,
\begin{multline*}
g_\alpha(x)~\sim~
\Big(\frac{(2\pi)^{r\alpha-1}}{\lambda\Gamma(r+1)}\Big)^{1/2}\Big(\frac{(1-\alpha)^{(r+2)\alpha-1}}{\alpha^{(r+2)\alpha-1}(\sigma_\alpha^2)^{(r+1)\alpha}}\Big)^{1/(2(1-\alpha))} \exp\Big(\frac{r}{2}
\Big(\Big(\frac{\alpha\sigma_\alpha^2}{1-\alpha}\Big)^{\alpha/(1-\alpha)}-1\Big)\Big) 
\\\times x^{((r+2)\alpha-1)/(2(1-\alpha))} \exp\Big(-\Big(\frac{1-\alpha}{(\alpha\sigma_\alpha^2)^\alpha}\Big)^{1/(1-\alpha)}(x-b\gamma_\alpha)^{1/(1-\alpha)}\Big) 
\end{multline*}
with 
$\sigma^2_\alpha$ as defined in Theorem \ref{thm:main1}.
\end{thm}
\begin{rem}
Of course, Theorems \ref{thm:main1} and \ref{thm:main2} can be deduced from Theorems \ref{thm:dens1} and \ref{thm:dens2} with the help of the L'H\^opital rule. However, we believe our present approach has its own merits and may be used for investigating precise distributional tail behavior of other infinite weighted sums of independent identically distributed random variables. Its main advantage is that absolute continuity of the distribution of a sum is not a priori required.
\end{rem}

\section{A look from a broader perspective}\label{sect:discussion}

We start by discussing the situation, in which $\mmp\{\eta>y\}>0$ for all $y>0$. 
Assume first that 
some exponential moments of positive orders are finite, that is,
$$
t_0:=\sup\{t>0:\me[\eee^{t\eta}]<\infty\}\in(0,\infty).
$$
Set $S'(\alpha):=\sum_{k\ge 2}k^{-\alpha}\eta_k$. It is clear that
$$
\me[\eee^{tS'(\alpha)}]<\infty
\quad\text{for all }~ 0\leq t<2^\alpha t_0.
$$
Therefore, the distribution tail of $S'(\alpha)$ is lighter than that of
$\eta$. This ensures 
that the distribution tail of $S(\alpha)$ is proportional to the distribution tail of $\eta$. We 
illustrate this observation 
under the assumption that, 
for some $a\in\mathbb{R}$, 
$$
\mmp\{\eta>x\}~\sim~ x^a\eee^{-t_0x},\quad x\to\infty.
$$
Indeed, by the total probability formula,
$$
\frac{\mmp\{S(\alpha)>x\}}{\mmp\{\eta 
>x\}}
=\int_\mathbb{R}\frac{\mmp\{\eta 
>x-y\}}{\mmp\{\eta 
>x\}}
\mmp\{S'(\alpha)\in {\rm d}y\}.
$$
Noting that
$\lim_{x\to\infty}\frac{\mmp\{\eta 
>x-y\}}{\mmp\{\eta 
>x\}}= \eee^{t_0y}$ for every fixed $y\in\mr$ and using Lebesgue's dominated convergence theorem, we conclude that
$$
\lim_{x\to\infty}
\frac{\mmp\{S(\alpha)>x\}}{\mmp\{\eta 
>x\}}
=\int_\mathbb{R}\eee^{t_0y}
\mmp\{S'(\alpha)\in {\rm d}y\}\in(0,\infty).
$$

A similar relation holds true 
in the case where the distribution tail of $\eta$ is heavy, that is, 
$t_0=0$. If, for example, the function $x\mapsto \mmp\{\eta>x\}$ is regularly varying at $\infty$ of index $-\theta$ for $\theta>2$, then
$$
\lim_{x\to\infty}
\frac{\mmp\{S(\alpha)>x\}}{\mmp\{\eta 
>x\}}
=\lim_{x\to\infty}
\frac{\mmp\{\eta_k>xk^{\alpha}\text{ for some }k\ge 1\}}{\mmp\{\eta 
>x\}}
=\sum_{k\geq 1} k^{-\alpha \theta 
}.
$$
These equalities can be proven with the help of arguments which are standard for heavy-tailed distributions. The driving force behind the aforementioned situations is 
the classical `one big jump strategy', according to which a large value of the sum
$S(\alpha)$ is caused by a large value of a single summand.

If the moment generating function $t\mapsto \me[\eee^{t\eta}]$ is finite for all $t>0$, 
it is still possible that the distribution tail of $S(\alpha)$ is of the same type as the distribution tail of $\eta$. Here is a rather simple example. 
Assume that $\eta$ is normally distributed with zero mean and unit variance. The sum 
$S(\alpha)$ is then also normally distributed with zero mean and variance
$\theta^2_\alpha:=\sum_{k\ge1}k^{-2\alpha}$. In particular, 
$$
\mmp\{S(\alpha)>x\}=\mmp\{\eta>x/\theta_\alpha\}, \quad x\in\mr. 
$$
Actually, a weaker version of this equality holds true 
for a much wider class of distributions. Assume that
\begin{equation}
\label{eq:mgf-log}
\psi(t)=\log\me[\eee^{t\eta}]\sim \gamma t^{p},\quad t\to\infty
\end{equation}
for some $\gamma>0$ and $p>1$.
If $\alpha>1/p$, then 
$$
\log\me[\eee^{tS(\alpha)}]
=\sum_{k\ge1}\psi(tk^{-\alpha})~\sim~ \gamma t^p\sum_{k\ge1}k^{-\alpha p},\quad t\to\infty.
$$
Applying now the Kasahara Tauberian theorem (see Theorem 4.12.7 in \cite{Bingham+Goldie+Teugels:1989}), we conclude that
\begin{equation}\label{eq:log-equiv}
\log\mmp\{S(\alpha)>x\}~\sim~\log\mmp\{\eta>x/(\sum_{k\ge1}k^{-\alpha p})^{1/p}\}, \quad x\to\infty.
\end{equation}
In other words, the distribution tails of $S(\alpha)$ and $\eta$ are proportional to each other on the logarithmic scale provided that $\alpha>1/p$. This proportionality disappears for
$\alpha\le 1/p$. More precise versions of \eqref{eq:log-equiv} (with the logarithms omitted) 
show up in the tail asymptotics of series $\sum_{k\ge 1}c_k\eta_k$ with $(c_k)_{k\geq 1}$ being a summable sequence.
This case has been investigated by Rootz\'en~\cite{Rootzen:1986,Rootzen:1987} and 
Kl\"uppelberg and Lindner~\cite{KlueppelbergLindner:2005}.

The proportionality arising in \eqref{eq:log-equiv} disappears for $\alpha\le 1/p$. Indeed, assume, for instance, that 
\eqref{eq:mgf-log} holds and that $\alpha<$ 
$1/p$. It can be checked along the lines of the proofs of Propositions~\ref{prop:first part1} and \ref{prop:first part2} below that   
$$
\log \me [\eee^{tS(\alpha)}] 
~\sim~\frac{t^{1/\alpha}}{\alpha}
\int_0^\infty v^{-1-1/\alpha}\psi(v){\rm d}v,\quad t\to\infty,
$$
but we omit details. 
Since $1/(1-\alpha)<p/(p-1)$, 
the Kasahara Tauberian theorem implies that
$$
-\log\mmp\{S(\alpha)>x\}~\sim~ C_1 x^{1/(1-\alpha)}
\ll C_2 x^{p/(p-1)}~\sim~ - 
\log\mmp\{\eta>x\},\quad x\to\infty
$$
for appropriate positive constants $C_1$ and $C_2$. The first limit relation 
agrees with the precise (rather than logarithmic) tail behavior exhibited in Theorems~\ref{thm:main1} and \ref{thm:main2} for $\alpha\in (1/2,1)$.

Plainly, 
relation \eqref{eq:mgf-log} is not sufficient for obtaining a precise asymptotic. To succeed,  
more 
information on the function $\psi$ is needed. Under our standing assumption 
$\eta\le b$ a.s., 
$\psi(t)\sim bt$ as $t\to\infty$. Further terms in the asymptotic expansion for $\psi$ depend on the behaviour of probabilities $\mmp\{\eta\in[b-\delta,b]\}$ as $\delta\to 0+$. This fact justifies our assumptions in Theorems~\ref{thm:main1} and \ref{thm:main2}. It will be clear from the proofs that our argument still applies whenever  
$$
\psi''(t)=t^{p-2}L(t)+o(t^{-2}),\quad t\to\infty
$$
for some $p\le1/\alpha$ and some $L$ slowly varying at infinity. 
Summarizing, the assumed 
boundedness of $\eta$ from the right admits a simple link between the asymptotics of  
$\mmp\{\eta\in[b-\delta,b]\}$ as $\delta\to 0+$ and 
$\psi(t)$ as $t\to\infty$ and does not lead to a significant restriction of generality.
\section{Description of our approach to the tail asymptotics}\label{sec:structure}

We first perform the 
exponential change of measure, which is standard in the area of large deviations. More precisely,
for each $t>0$, we define a new probability measure $\mmp^{(t)}$ by
\begin{equation}\label{eq:change2}
\me^{(t)}[g(\eta_1,\ldots, \eta_k)]=\frac{\me [\eee^{tS(\alpha)}g(\eta_1,\ldots,\eta_k)]}{\me [\eee^{tS(\alpha)}]},
\end{equation}
where $\me^{(t)}$ denotes expectation with respect to $\mmp^{(t)}$. The equality is assumed to hold for all $k\in\mn$ and each bounded Borel function $g:\mr^k\to\mr$. Let $f:\mr\to\mr$ be any bounded Borel function. Putting $g(y_1,\ldots, y_k)=f\big(\sum_{j=1}^k j^{-\alpha} y_j\big)$ for $(y_1,\ldots, y_k)\in\mr^k$ and letting $k\to\infty$ we infer
\begin{equation}\label{eq:change1}
\me^{(t)}[f(S(\alpha))]=\frac{\me [\eee^{tS(\alpha)}f(S(\alpha))]}{\me [\eee^{tS(\alpha)}]}.
\end{equation}
Equality \eqref{eq:change1} holds true for any Borel function $f:\mr\to\mr$ which is not necessarily bounded, whenever the left- or right-hand side of \eqref{eq:change1} is well-defined, possibly infinite.

Recall the notation $\psi(t)=\log \me [ \eee^{t\eta}]$ for $t\in\mr$. Then $\me [\eee^{tS(\alpha)}]=\exp(\sum_{k\geq 1}\psi(t/k^\alpha))$ for $t\in\mr$. Using \eqref{eq:change1} with $f(y)=y$ we obtain $$\me^{(t)}[S(\alpha)]=\frac{\me [\eee^{tS(\alpha)}S(\alpha)]}{\me [\eee^{tS(\alpha)}]}=\sum_{k\geq 1}k^{-\alpha}\psi^\prime(t/k^\alpha).$$ Fix any $x>0$ and put $f(y)=\me \eee^{-tS(\alpha)}\1_{(x,\infty)}(y)$. Then \eqref{eq:change1} reads $$\mmp\{S(\alpha)>x\}=\me [\eee^{tS(\alpha)}]\me^{(t)}[\eee^{-tS(\alpha)}\1_{\{S(\alpha)>x\}}]=\eee^{-tx+\sum_{k\geq 1}\psi(t/k^\alpha)}\me^{(t)}[\eee^{-t(S(\alpha)-x)}\1_{\{S(\alpha)-x>0\}}].$$ Under $\mmp^{(t)}$, put $S_0^{(t)}(\alpha)=S(\alpha)-\me^{(t)}[S(\alpha)]$. We shall write $S_0(\alpha)$ for $S_0^{(t)}(\alpha)$ unless it leads to ambiguity. The function $t\mapsto \sum_{k\geq 1}k^{-\alpha}\psi^\prime(t/k^\alpha)$ is continuous, strictly increasing on $[0,\infty)$ and equal to $0$ at $0$. Hence, for each $x\geq 0$, the equation
\begin{equation}
\label{eq:t(x)}
\sum_{k\geq 1}k^{-\alpha}\psi^\prime(t/k^\alpha)=x
\end{equation}
has a unique solution $t=t(x)$. We shall investigate the asymptotic behavior of $\mmp\{S(\alpha)>x\}$ as $x\to\infty$ with the help of the representation
\begin{multline}\label{eq:repr}
\mmp\{S(\alpha)>x\}\\=\exp\Big\{\sum_{k\geq 1}(\psi(t(x)/k^\alpha)-(t(x)/k^\alpha)\psi^\prime(t(x)/k^\alpha))\Big\}\me^{(t(x))}[\eee^{-t(x)S_0(\alpha)}\1_{\{S_0(\alpha)>0\}}].
\end{multline}

Our analysis of the terms on the right-hand 
side of \eqref{eq:repr} consists of the following four steps.

\noindent {\sc Step 1}. Find an asymptotic expansion of
$\me^{(t)}[S(\alpha)]=\sum_{k\geq 1}k^{-\alpha}\psi^\prime(t/k^\alpha)$ as $t\to\infty$ with a sufficient precision. It turns out that an expansion up to the term $o(1/t)$ serves our purpose. 
Use the 
expansion to determine the asymptotic behaviour of the solution $t(x)$ to 
equation \eqref{eq:t(x)}. This step is implemented in Propositions \ref{prop:asymp1} and \ref{prop:asymp2}.

\noindent {\sc Step 2}. Find an asymptotic expansion of $\sum_{k\geq 1}(\psi(t/k^\alpha)-(t/k^\alpha)\psi^\prime(t/k^\alpha))$ as $t\to\infty$ up to the term $o(1)$. The $o(1)$-precision is needed to obtain precise asymptotics for the exponential term in \eqref{eq:repr}. This is done in Propositions \ref{prop:first part1} and \ref{prop:first part2}.

\noindent {\sc Step 3}. Find the first-order asymptotic of $\me^{(t)}[\eee^{-tS_0(\alpha)}\1_{\{S_0(\alpha)>0\}}]$ as $t\to\infty$. To this end we prove a local Central Limit Theorem (CLT) for $S_0(\alpha)$ under $\mmp^{(t)}$. We note 
that the use of CLT-like results is also a very common tool in the study of large deviation probabilities. For example, 
one applies the Berry-Esseen inequality to derive an exact large deviation asymptotic 
for sums of independent random variables, see Chapter VIII in Petrov's book \cite{Petrov:1975}. Surprisingly, it turned out that the application of the Berry-Esseen inequality to
$S(\alpha)$ does not allow one to determine the asymptotics of $\me^{(t)}[\eee^{-tS_0(\alpha)}\1_{\{S_0(\alpha)>0\}}]$.
A further peculiarity of $S(\alpha))$ consists in the fact that its variance under $\mmp^{(t)}$ goes to zero and, consequently, one has a kind of 'superconcentration' effect around $\me^{(t)}[S(\alpha)]$. \newline
The local CLT is proven in Theorem~\ref{thm:Stone}. The expectation
$\me^{(t)}[\eee^{-tS_0(\alpha)}\1_{\{S_0(\alpha)>0\}}]$ is analysed in Proposition \ref{prop:second part}.

\noindent {\sc Step 4}. To conclude, replace $t$ with $t(x)$ in the asymptotic expansions obtained at Steps 2 and 3. In this way both Theorems \ref{thm:main1} and \ref{thm:main2} follow.

\section{Preparation for the proofs}

\subsection{A version of the Euler-Maclaurin formula}

There are several versions of the Euler-Maclaurin summation formula. Below we state the one that serves our needs.

Let $m,n\in\mn$, $m<n$ and $f: [m,n]\to \mr$ be a twice continuously differentiable function. A specialization of formula (9.78) on p.~460 in \cite{Graham+Knuth+Patashnik:1990} yields
\begin{equation}\label{eq:EulMas finite}
\sum_{j=m}^n f(j)=\int_m^n f(x){\rm d}x+(f(n)+f(m))/2+(f^\prime(n)-f^\prime(m))/12+R_{m,n},
\end{equation}
where $|R_{m,n}|\leq (1/12)\int_m^n |f^{\prime\prime}(x)|{\rm d}x$.

If $f: [m,\infty)\to \mr$ is a twice continuously differentiable function with $\lim_{x\to\infty}f(x)=\lim_{x\to\infty}f^\prime(x)=0$ and $\int_m^\infty |f^{\prime\prime}(x)|{\rm d}x<\infty$, then
\begin{equation}\label{eq:EulMas infinite}
\sum_{j\geq m} f(j)=\int_m^\infty f(x){\rm d}x+ (f(m))/2- (f^\prime(m))/12+R_m,
\end{equation}
where $|R_{m}|\leq (1/12)\int_m^\infty |f^{\prime\prime}(x)|{\rm d}x$.

\subsection{Auxiliary results}


\begin{lemma}\label{lem:aux1}
Assume that $\eta\leq b$ a.s.\ and $\mmp\{\eta=b\}=\theta\in (0,1)$. Then, as $t\to\infty$,

\noindent (a) $\psi(t)=bt+\log \theta+o(1)$;

\noindent (b) $t\psi^\prime(t)=bt+o(1)$ and $L^\prime(t)=o(1/t)$, where $L(t):=\log \me [\eee^{-t(b-\eta)}]$ for $t\geq 0$;

\noindent (c) $\lim_{t\to\infty} t^2 \psi^{\prime\prime}(t)=\lim_{t\to\infty} t^3 |\psi^{\prime\prime\prime}(t)|=0$.
\end{lemma}
\begin{proof}
(a) This is justified as follows $$\psi(t)=bt+\log \me\big[\eee^{-t(b-\eta)}\big]=bt+\log \theta+\log\big(1+\theta^{-1}\me\big[\eee^{-t(b-\eta)}\1_{\{\eta<b\}}\big]\big)=bt+\log\theta+o(1),\quad t\to\infty.$$ The last equality stems 
from the fact that $\lim_{t\to\infty}\me \big[\eee^{-t(b-\eta)}\1_{\{\eta<b\}}\big]=0$.

\noindent (b) Put $\ell(t):=\me [\eee^{-t(b-\eta)}]$ for $t\geq 0$. Then $t\psi^\prime(t)=bt+t\ell^\prime(t)/\ell(t)=bt+tL^\prime(t)$. By Lebesgue's dominated convergence theorem, for $n\in\mn$,
\begin{equation}\label{eq:conv to zero}
t^n\big|\ell^{(n)}(t)\big|=\me [(t(b-\eta))^n \eee^{-t(b-\eta)}]~\to~ 0,\quad t\to\infty,
\end{equation}
where $\ell^{(n)}$ denotes the $n$th derivative of the function $\ell$. Indeed, $\lim_{t\to\infty}(t(b-\eta))^n \eee^{-t(b-\eta)}=0$ a.s., and the function $x\mapsto x^n\eee^{-x}$ is bounded on $[0,\infty)$. Since
\begin{equation}\label{eq:gamma}
\lim_{t\to\infty}\ell(t)=\theta,
\end{equation}
the claims of part (b) follow from \eqref{eq:conv to zero} with $n=1$.

\noindent (c) The proof is analogous to that of part (b). We only treat the third derivative. Since
\begin{equation}\label{eq:third}
\psi^{\prime\prime\prime}(t)=\frac{\ell^{\prime\prime\prime}(t)\ell(t)-\ell^\prime(t)\ell^{\prime\prime}(t)}{(\ell(t))^2}-\frac{2\ell^\prime(t)}{\ell(t)}\frac{\ell^{\prime\prime}(t)\ell(t)-(\ell^\prime(t))^2}{(\ell(t))^2},
\end{equation}
the result is secured by 
\eqref{eq:conv to zero} with $n=1,2,3$ and \eqref{eq:gamma}.
\end{proof}

\begin{lemma}\label{lem:aux2}
Assume that $\eta\leq b$ a.s.\ and $\mmp\{b-\eta\leq x\}\sim \lambda x^r$ as $x\to 0+$ for positive $\lambda$ and $r$. Then, as $t\to\infty$,

\noindent (a) $\psi(t)=bt-r\log t+\log\left(\lambda\Gamma(r+1)\right)+o(1)$;

\noindent (b) $L^\prime(t)=-rt^{-1}+o(1/t)$ and $\psi^\prime(t)=b-rt^{-1}+o(1/t)$;

\noindent (c) $\psi^{\prime\prime}(t)=L^{\prime\prime}(t)=rt^{-2}+o(1/t^2)$ and $\psi^{\prime\prime\prime}(t)=L^{\prime\prime\prime}(t)=-2rt^{-3}+o(1/t^3)$.
\end{lemma}
\begin{proof}
(a) By Theorem 1.7.1' in \cite{Bingham+Goldie+Teugels:1989},
\begin{equation}\label{eq:relat1}
\ell(t)=\me \eee^{-t(b-\eta)}~\sim~\lambda \Gamma(1+r)t^{-r},\quad t\to\infty,
\end{equation}
where $\Gamma$ is the Euler gamma function. This entails
$$
L(t)=\log \ell(t)=-r\log t+\log\left(\lambda\Gamma(r+1)\right)+o(1),
$$
whence
$$
\psi(t)=bt+L(t)=bt-r\log t+\log\left(\lambda\Gamma(r+1)\right)+o(1).
$$

\noindent (b) Using $\ell(t)=\int_t^\infty (-\ell^\prime(x)){\rm d}x$, the fact that $-\ell'$ is nonincreasing and the monotone density theorem (Theorem 1.7.2 in \cite{Bingham+Goldie+Teugels:1989}) we infer
\begin{equation}\label{eq:relat2}
-\ell^\prime(t)~\sim~\lambda r\Gamma(1+r)t^{-(1+r)},\quad t\to\infty.
\end{equation}
Relations \eqref{eq:relat1} and \eqref{eq:relat2} entail $$L^\prime(t)=rt^{-1}+o(1/t),\quad t\to\infty.$$

The proof of part (c) is analogous, hence omitted. The basic observation is that, for $n\geq 2$, $(-1)^n \ell^{(n)}$ is a nonincreasing function. This enables us to use the monotone density theorem.
\end{proof}

\begin{lemma}\label{lem:at zero}
Under the sole assumptions $\me[\eta]=0$ and $\me [\eta^2]\in (0,\infty)$, as $t\to0$,

$$
\psi(t)~\sim~\me [\eta^2]t^2/2,
\ \ \psi^\prime(t)~\sim~\me [\eta^2] t,\quad \text{and} \ \ \psi^{\prime\prime}(t)\sim\me \eta^2.
$$
If, in addition, $\eta\leq b$ a.s., then $$\lim_{t\to 0+}t \psi^{\prime\prime\prime}(t)=0.$$
\end{lemma}
\begin{proof}
The first three claims are standard, and we omit a proof.

As for the last limit relation, observe that, in view of \eqref{eq:third}, the limit $\lim_{t\to 0+}\psi^{\prime\prime\prime}(t)$ is finite provided that $\me [|\eta|^3]<\infty$. Thus, the last claim holds trivially in this case. If $\me [|\eta|^3]=\infty$, then, as $t\to 0+$, the limits of all functions, except $\ell^{\prime\prime\prime}$, appearing in \eqref{eq:third} are still finite. It remains to note that, by Lebesgue's dominated convergence theorem,
\begin{equation*}
t\big|\ell^{\prime\prime\prime}(t)\big|=\me [(b-\eta)^2(t(b-\eta))\eee^{-t(b-\eta)}]~\to~ 0,\quad t\to 0+
\end{equation*}
because the function $x\mapsto x\eee^{-x}$ is bounded on $[0,\infty)$ and $\me [(b-\eta)^2]<\infty$.

\end{proof}

\section{Proofs of Theorems \ref{thm:main1} and \ref{thm:main2} 
}

We follow the steps outlined in Section \ref{sec:structure}. Step 1 is realized by the following propositions.
\begin{assertion}\label{prop:asymp1}
Assume that $\eta\leq b$ a.s.\ with $\mmp\{\eta=b\}=\theta\in (0,1)$, $\me[\eta]=0$ and $\me[\eta^2]<\infty$. If $\alpha=1$, then
\begin{equation}\label{eq:formula 1}
\sum_{k\geq 1}k^{-1}\psi^\prime(t/k)=b\log t+q+o(1/t),\quad t\to\infty,
\end{equation}
where $q=b\gamma_1+\int_0^1 x^{-1}\psi^\prime(x){\rm d}x+\int_1^\infty x^{-1}(\psi^\prime(x)-b){\rm d}x\in\mr$. For each $x\geq 0$, the equation $\sum_{k\geq 1}k^{-1}\psi^\prime(t/k)=x$ has a unique solution $t=t(x)$ satisfying
\begin{equation}\label{eq:expansion}
t(x)=\exp((x-q)/b)+o(1),\quad x\to\infty.
\end{equation}
If $\alpha\in (1/2, 1)$, then
\begin{equation}\label{eq:formula 2}
\sum_{k\geq 1}k^{-\alpha}\psi^\prime(t/k^\alpha)=r_\alpha t^{-1+1/\alpha}+b\gamma_\alpha 
+o(1/t),\quad t\to\infty,
\end{equation}
where $r_\alpha:=\alpha\sigma^2_\alpha/(1-\alpha) 
\in (0,\infty)$. 
For each $x\geq 0$, the equation $\sum_{k\geq 1}k^{-1}\psi^\prime(t/k)=x$ has a unique solution $t=t(x)$ satisfying
\begin{equation}\label{eq:expansion2}
t(x)=(r_\alpha^{-1}(x-b\gamma_\alpha 
)+o(x^{-\alpha/(1-\alpha)}))^{\alpha/(1-\alpha)},\quad x\to\infty.
\end{equation}
\end{assertion}
\begin{assertion}\label{prop:asymp2}
Assume that $\eta\leq b$ a.s., $\mmp\{b-\eta\leq x\}\sim \lambda x^r$ as $x\to 0+$ for positive $\lambda$ and $r$, $\me[\eta]=0$ and $\me [\eta^2] 
<\infty$. If $\alpha=1$, then
\begin{equation}\label{eq:formula 11}
\sum_{k\geq 1}k^{-1}\psi^\prime(t/k)=b\log t+q+rt^{-1}/2+o(1/t),\quad t\to\infty
\end{equation}
with the same $q$ as in Proposition \ref{prop:asymp1}. For each $x\geq 0$, the equation $\sum_{k\geq 1}k^{-1}\psi^\prime(t/k)=x$ has a unique solution $t=t(x)$ satisfying
\begin{equation}\label{eq:expansion11}
t(x)=\exp((x-q)/b)-r/(2b)+o(1),\quad x\to\infty.
\end{equation}
If $\alpha\in (1/2, 1)$, then
\begin{equation}\label{eq:formula 21}
\sum_{k\geq 1}k^{-\alpha}\psi^\prime(t/k^\alpha)=r_\alpha t^{-1+1/\alpha}+b\gamma_\alpha 
+rt^{-1}/2+o(1/t),\quad t\to\infty
\end{equation}
with the same $r_\alpha$ 
as in Proposition \ref{prop:asymp1}. 
For each $x\geq 0$, the equation $\sum_{k\geq 1}k^{-1}\psi^\prime(t/k)=x$ has a unique solution $t=t(x)$ satisfying
\begin{equation}\label{eq:expansion22}
t(x)=(r_\alpha^{-1}(x-b\gamma_\alpha 
)-(rr_\alpha^{(2\alpha-1)/(1-\alpha)}/2)x^{-\alpha/(1-\alpha)}+o(x^{-\alpha/(1-\alpha)}))^{\alpha/(1-\alpha)},\quad x\to\infty.
\end{equation}
\end{assertion}

Now we are passing to Step 2.
\begin{assertion}\label{prop:first part1}
Assume that $\eta\leq b$ a.s.\ with $\mmp\{\eta=b\}=\theta\in (0,1)$, $\me[\eta]=0$ and $\me [\eta^2] 
<\infty$. If $\alpha=1$, then
\begin{equation}\label{eq:first 1}
\sum_{k\geq 1}\big(\psi(t/k)-(t/k)\psi^\prime(t/k)\big)=-bt-2^{-1}\log\theta+o(1),\quad t\to\infty,
\end{equation}
whereas if $\alpha\in (1/2, 1)$, then
$$\sum_{k\geq 1}\big(\psi(t/k^\alpha)-(t/k^\alpha)\psi^\prime(t/k^\alpha)\big)=-\alpha^{-2}(1-\alpha)\kappa_\alpha t^{1/\alpha}-2^{-1}\log\theta+o(1),\quad t\to\infty,$$ where $\kappa_\alpha=\int_0^\infty x^{-1-1/\alpha}\psi(x){\rm d}x<\infty$.
\end{assertion}
\begin{rem}\label{rem:cont}
At the first glance it may seem that the asymptotic is `discontinuous' at $\alpha=1$. However, this is not the case. We shall show below that
\begin{equation}\label{eq:cont}
\lim_{\alpha\to 1-0}(1-\alpha)\kappa_\alpha=b.
\end{equation}
\end{rem}

\begin{assertion}\label{prop:first part2}
Assume that $\eta\leq b$ a.s., $\mmp\{b-\eta\leq x\}\sim \lambda x^r$ as $x\to 0+$ for positive $\lambda$ and $r$, $\me[\eta]=0$ and $\me [\eta^2] 
<\infty$. If $\alpha=1$, then, as $t\to\infty$,
\begin{multline}\label{eq:first 11}
\sum_{k\geq 1}\big(\psi(t/k)-(t/k)\psi^\prime(t/k)\big)=-bt+(r/2)\log t+(r/2)(\log (2\pi)-1)\\-(1/2)\log(\lambda\Gamma(r+1))+o(1),
\end{multline}
whereas if $\alpha\in (1/2, 1)$, then
\begin{align}\label{eq:first 21}
\nonumber
\sum_{k\geq 1}\big(\psi(t/k^\alpha)-(t/k^\alpha)\psi^\prime(t/k^\alpha)\big)&=-\alpha^{-2}(1-\alpha)\kappa_\alpha t^{1/\alpha}+(r/2)\log t\\
&\hspace{0.5cm}+(r/2)(\alpha \log (2\pi)-1 
)
-(1/2)\log(\lambda\Gamma(r+1))+o(1),
\end{align}
where $\kappa_\alpha=\int_0^\infty x^{-1-1/\alpha}\psi(x){\rm d}x<\infty$.
\end{assertion}

Step 3 is implemented with the help of the result given next.
\begin{assertion}\label{prop:second part}
Assume that $\eta\leq b$ a.s., $\me[\eta]=0$ and
$\me[\eta^2]
\in (0,\infty)$. Then
$$
\lim_{t\to\infty} t^{1/(2\alpha)}\me^{(t)}\big[\eee^{-tS_0(\alpha)}\1_{\{S_0(\alpha)>0\}}\big]=(2\pi\sigma_\alpha^2)^{-1/2},
$$
where
\begin{equation}\label{eq:sigma}
\sigma^2_\alpha=\alpha^{-1}\int_0^\infty x^{1-1/\alpha} \psi^{\prime\prime}(x){\rm d}x\in (0,\infty)
\end{equation}
(in particular, $\sigma_1^2=b$).
\end{assertion}

With these propositions at hand, we are ready to prove Theorems \ref{thm:main1} and \ref{thm:main2}.
\begin{proof}[Proof of Theorem \ref{thm:main1}]
We only discuss the case $\alpha\in (1/2,1)$, the case $\alpha=1$ being simpler.

Our starting point is representation \eqref{eq:repr}. Invoking Proposition \ref{prop:second part} in combination with $\lim_{x\to\infty}t(x)=+\infty$ we obtain $$\me^{(t(x))}\big[\eee^{-t(x)S_0(\alpha)}\1_{\{S_0(\alpha)>0\}}\big]~\sim~\frac{1}{(2\pi \sigma^2_\alpha (t(x))^{1/\alpha})^{1/2}},\quad x\to\infty.$$ Using \eqref{eq:expansion2} we infer
\begin{multline}
(t(x))^{1/\alpha}-(r^{-1}_\alpha(x-b\gamma_\alpha 
))^{1/(1-\alpha)}=(r^{-1}_\alpha(x-b\gamma_\alpha 
)+o(x^{-\alpha/(1-\alpha)}))^{1/(1-\alpha)}-(r^{-1}_\alpha(x-b\gamma_\alpha 
))^{1/(1-\alpha)}\\=O(x^{1/(1-\alpha)})o(x^{-1/(1-\alpha)})=o(1),\quad x\to\infty.\label{eq:refine}
\end{multline}
This yields
\begin{equation}\label{eq:second factor}
\me^{(t(x))}\big[\eee^{-t(x)S_0(\alpha)}\1_{\{S_0(\alpha)>0\}}\big]~\sim~\Big(\frac{r^{1/(1-\alpha)}_\alpha}{2\pi \sigma^2_\alpha x^{1/(1-\alpha)}}\Big)^{1/2},\quad x\to\infty.
\end{equation}
By Proposition \ref{prop:first part1}, \eqref{eq:refine} and noting that $(1-\alpha)\alpha^{-2}\kappa_\alpha=\alpha\sigma^2_\alpha$,
\begin{multline*}
\exp\Big(\sum_{k\geq 1}(\psi(t(x)/k^\alpha)-(t(x)/k^\alpha)\psi^\prime(t(x)/k^\alpha))\Big)~\sim~
\frac{1}{\theta^{1/2}}\exp\Big(-\frac{(1-\alpha)\kappa_\alpha}{\alpha^2}(t(x))^{1/\alpha}\Big)\\~\sim~
\frac{1}{\theta^{1/2}}\exp\Big(-\alpha \sigma^2_\alpha\Big(\frac{x-b\gamma_\alpha 
}{r_\alpha}\Big)^{1/(1-\alpha)}\Big),\quad x\to\infty.
\end{multline*}
Combining this with \eqref{eq:second factor} proves Theorem \ref{thm:main1} in the case $\alpha\in (1/2,1)$.
\end{proof}

\begin{proof}[Proof of Theorem \ref{thm:main2}]
This proof is analogous to that of Theorem \ref{thm:main1}. We only give a counterpart of \eqref{eq:refine} for $t=t(x)$ satisfying \eqref{eq:expansion22}:
\begin{equation*}
(t(x))^{1/\alpha}=(r^{-1}_\alpha(x-b\gamma_\alpha 
))^{1/(1-\alpha)}-(rr_\alpha^{(2\alpha-1)/(1-\alpha)})/(2(1-\alpha))+o(1),\quad x\to\infty.
\end{equation*}
\end{proof}

\subsection{Proof of Propositions \ref{prop:asymp1} and \ref{prop:asymp2}}

\begin{proof}[Proof of Proposition \ref{prop:asymp1}]
Fix any $\beta>2/(2\alpha-1)>1$. By Lemma \ref{lem:at zero}, $\psi^\prime(t)\sim \me [\eta^2]t$ as $t\to 0$. Hence, $$\sum_{k\geq \lfloor t^\beta\rfloor+1}k^{-\alpha}\psi^\prime(t/k^\alpha)=O\Big(t\sum_{k\geq \lfloor t^\beta\rfloor+1}k^{-2\alpha}\Big)=O(t^{1-(2\alpha-1)\beta})=o(1/t),\quad t\to\infty.$$ Recall that $\psi^\prime(t)=b+L^\prime(t)$, where $L(t)=\log \me [\eee^{-t(b-\eta)}]$ for $t\geq 0$, and write $$\sum_{k=1}^{\lfloor t^\beta\rfloor}k^{-\alpha}\psi^\prime(t/k^\alpha)=b\sum_{k=1}^{\lfloor t^\beta\rfloor}k^{-\alpha}+\sum_{k=1}^{\lfloor t^\beta\rfloor}k^{-\alpha}L^\prime(t/k^\alpha):=A_\alpha(t)+B_\alpha(t).$$ By Theorem 3.2 (a,b) on p.~55 in \cite{Apostol:1976}, in the case $\alpha=1$,
\begin{equation}\label{eq:at}
A_1(t)=b\log (\lfloor t^\beta\rfloor)+b\gamma_1+O(t^{-\beta}),\quad t\to\infty,
\end{equation}
whereas in the case $\alpha\in (1/2,1)$,
\begin{equation}\label{eq:at2}
A_\alpha(t)=b(1-\alpha)^{-1}(\lfloor t^\beta\rfloor)^{1-\alpha}+b\gamma_\alpha+O(t^{-\alpha\beta}),\quad t\to\infty,
\end{equation}
where $\gamma_\alpha$ is the Euler-Mascheroni constant, see \eqref{eq:EulMas<1} and \eqref{eq:EulMas=1}.

We intend to use formula \eqref{eq:EulMas finite} with $f=f_t$, $m=1$ and $n=\lfloor t^\beta\rfloor$, where $f_t(x)=x^{-\alpha}L^\prime(t/x^\alpha)$. For later needs, we note that $$f_t^\prime(x)=-\alpha\big(x^{-(\alpha+1)}L^\prime(t/x^\alpha)+(t/x^{2\alpha+1})L^{\prime\prime}(t/x^\alpha)\big)$$ and  $$f_t^{\prime\prime}(x)=\alpha\big((\alpha+1)x^{-(\alpha+2)}L^\prime(t/x^\alpha)+(3\alpha+1)tx^{-(2\alpha+2)}L^{\prime\prime}(t/x^\alpha)+\alpha t^2 x^{-(3\alpha+2)}L^{\prime\prime\prime}(t/x^\alpha)\big).$$  According to \eqref{eq:EulMas finite},
\begin{multline}
B_\alpha(t)=\int_1^{\lfloor t^\beta\rfloor}x^{-\alpha}L^\prime(t/x^\alpha){\rm d}x+\big(L^\prime(t)+(\lfloor t^\beta\rfloor)^{-\alpha}L^\prime(t/(\lfloor t^\beta\rfloor)^\alpha)\big)/2\\+\alpha\big(L^\prime(t)+tL^{\prime\prime}(t)-(\lfloor t^\beta\rfloor)^{-(\alpha+1)}L^\prime(t/(\lfloor t^\beta\rfloor)^\alpha)-(t/(\lfloor t^\beta\rfloor)^{2\alpha+1})L^{\prime\prime}\big(t/(\lfloor t^\beta\rfloor)^\alpha\big)\big)/12+R(t), \label{eq:inter4}
\end{multline}
where $R(t)\leq (1/12)\int_1^{\lfloor t^\beta \rfloor}|f_t^{\prime\prime}(x)|{\rm d}x$.

By Lemma \ref{lem:aux1}(b,c), $L^\prime(t)=o(1/t)$ and $tL^{\prime\prime}(t)=t\psi^{\prime\prime}(t)=o(1/t)$
as $t\to\infty$. Using $\lim_{u\to 0+}L^\prime(u)=-b$ and the choice of $\beta$ which ensures $\alpha\beta\geq (2\alpha-1)\beta>2>1$ because $\alpha\in (1/2,1]$, we infer $$(\lfloor t^\beta\rfloor)^{-\alpha}L^\prime\big(t/(\lfloor t^\beta\rfloor)^\alpha\big)=O(t^{-\alpha\beta})=o(1/t),\quad t\to\infty.$$ The latter trivially implies that $(\lfloor t^\beta\rfloor)^{-(\alpha+1)}L^\prime\big(t/(\lfloor t^\beta\rfloor)^\alpha\big)=o(1/t)$. According to Lemma \ref{lem:at zero}, $$(t/(\lfloor t^\beta\rfloor)^{2\alpha+1})L^{\prime\prime}\big(t/(\lfloor t^\beta\rfloor)^\alpha\big)=(t/(\lfloor t^\beta\rfloor)^{2\alpha+1})\psi^{\prime\prime}\big(t/(\lfloor t^\beta\rfloor)^\alpha\big)=O(t^{-\beta(2\alpha+1)-1})=o(1/t).$$
Further, substituting $y=t/x^\alpha$, applying the L'H\^opital rule and Lemma \ref{lem:aux1}(b), we obtain
\begin{align*}
\alpha
\int_1^\infty x^{-(\alpha+2)}L^\prime(t/x^\alpha){\rm d}x
&=t^{-(1+1/\alpha)}\int_1^\infty (t/x^\alpha)^{1/\alpha} (\alpha t/x^{\alpha+1})L^\prime(t/x^\alpha){\rm d}x\\
&=t^{-(1+1/\alpha)}\int_0^t y^{1/\alpha}L^\prime(y){\rm d}y~\sim~\alpha(\alpha+1)^{-1}L^\prime(t)\\
&=o(1/t),\quad t\to\infty.
\end{align*}
Similarly,
\begin{align*}
\alpha t\int_1^\infty x^{-(2\alpha+2)}L^{\prime\prime}(t/x^\alpha){\rm d}x
&=t^{-(1+1/\alpha)}\int_1^\infty (t/x^\alpha)^{1+1/\alpha} (\alpha t/x^{\alpha+1})L^{\prime\prime}(t/x^\alpha){\rm d}x\\
&=t^{-(1+1/\alpha)}\int_0^t y^{1+1/\alpha}L^{\prime\prime}(y){\rm d}y~\sim~\alpha(\alpha+1)^{-1}tL^{\prime\prime}(t)\\
&=o(1/t),\quad t\to\infty
\end{align*}
and
\begin{align*}
\alpha t^2\int_1^\infty x^{-(3\alpha+2)}L^{\prime\prime}(t/x^\alpha){\rm d}x
&=
t^{-(1+1/\alpha)}\int_1^\infty (t/x^\alpha)^{2+1/\alpha} (\alpha t/x^{\alpha+1})L^{\prime\prime\prime}(t/x^\alpha){\rm d}x\\
&=t^{-(1+1/\alpha)}\int_0^t y^{2+1/\alpha}L^{\prime\prime\prime}(y){\rm d}y~\sim~\alpha(\alpha+1)^{-1}t^2L^{\prime\prime\prime}(t)\\
&=o(1/t),\quad t\to\infty.
\end{align*}
This proves $R(t)=o(1/t)$. As a result, 
\begin{equation}\label{eq:bt}
B_\alpha(t)=\int_1^{\lfloor t^\beta\rfloor}x^{-\alpha}L^\prime(t/x^\alpha){\rm d}x+o(1/t),\quad t\to\infty.
\end{equation}

By Lemma \ref{lem:at zero}, $\lim_{y\to 0+}y^{-1}\psi^\prime(y)=\me [\eta^2]$. Hence, the integral $c_0(\alpha):=\int_0^1 y^{-1/\alpha}\psi^\prime(y){\rm d}y$ is well-defined and $$\int_0^{t/(\lfloor t^\beta\rfloor)^\alpha}y^{-1/\alpha}\psi^\prime(y){\rm d}y=O(t^{-(\beta \alpha-1)(2-1/\alpha)})=o(t^{-1/\alpha}),\quad t\to\infty$$ (indeed, $(\beta \alpha-1)(2-1/\alpha)>1/\alpha$ by the choice of $\beta$). In view of $L^\prime(y)=o(1/y)$ as $y\to\infty$ (see Lemma \ref{lem:aux1}(b)), we conclude that the integral $c_1(\alpha):=-\int_1^\infty y^{-1/\alpha}L^\prime(y){\rm d}y=-\int_1^\infty y^{-1/\alpha}(\psi^\prime(y)-b){\rm d}y$ is also well-defined and $-\int_t^\infty y^{-1/\alpha}L^\prime(y){\rm d}y=o(t^{-1/\alpha})$ as $t\to\infty$.
Finally,
\begin{multline}
\int_1^{\lfloor t^\beta\rfloor}x^{-\alpha}L^\prime(t/x^\alpha){\rm d}x\\
=\alpha^{-1}t^{-1+1/\alpha}\left(\int_{t/(\lfloor t^\beta\rfloor)^\alpha}^1 y^{-1/\alpha}L^\prime(y){\rm d}y+\int_1^t y^{-1/\alpha}L^\prime(y){\rm d}y\right)\\
=\alpha^{-1}t^{-1+1/\alpha}\left(-b\int_{t/(\lfloor t^\beta\rfloor)^\alpha}^1 y^{-1/\alpha}{\rm d}y+\int_0^1 y^{-1/\alpha}\psi^\prime(y){\rm d}y-\int_0^{t/(\lfloor t^\beta\rfloor)^\alpha}y^{-1/\alpha}\psi^\prime(y){\rm d}y\right.\\
\left.+\int_1^\infty y^{-1/\alpha}L^\prime(y){\rm d}y-\int_t^\infty y^{-1/\alpha}L^\prime(y){\rm d}y\right).\label{eq:inter5}
\end{multline}
In the case $\alpha=1$ we infer $$\int_1^{\lfloor t^\beta\rfloor}x^{-1}L^\prime(t/x){\rm d}x=b\log t-b\log (\lfloor t^\beta\rfloor)+c_0(1)-c_1(1)+o(1/t),\quad t\to\infty.$$
A combination of this, \eqref{eq:at} and \eqref{eq:bt} yields \eqref{eq:formula 1}. In the case $\alpha\in (1/2,1)$ we obtain $$\int_1^{\lfloor t^\beta\rfloor}x^{-\alpha}L^\prime(t/x^\alpha){\rm d}x=(b(1-\alpha)^{-1}+\alpha^{-1}(c_0(\alpha)-c_1(\alpha)))t^{-1+1/\alpha}-b(1-\alpha)^{-1}(\lfloor t^\beta\rfloor)^{1-\alpha} 
+o(1/t)$$ as $t\to\infty$. Recall that $\sigma^2_\alpha=(1-\alpha)\alpha^{-3}\int_0^\infty y^{-1-1/\alpha}\psi(y){\rm d}y$. This entails
\begin{align*}
b(1-\alpha)^{-1}+\alpha^{-1}(c_0(\alpha)-c_1(\alpha))
&=\alpha^{-1} 
\int_0^\infty y^{-1/\alpha}\psi^\prime(y){\rm d}y 
=\alpha^{-2}\int_0^\infty y^{-1-1/\alpha}\psi(y){\rm d}y 
\\
&\hspace{1cm}=\alpha\sigma_\alpha^2/(1-\alpha)  
=r_\alpha.
\end{align*}
Using these together with \eqref{eq:at2} and \eqref{eq:bt} we arrive at \eqref{eq:formula 2}. 

It was explained in Section \ref{sec:structure} that, for each $x\geq 0$, the equation $\sum_{k\geq 1}k^{-\alpha}\psi^\prime(t/k^\alpha)=x$ has a unique solution $t=t(x)$. To determine its asymptotic behaviour in the case $\alpha=1$, write
\begin{equation}\label{eq:original}
b\log t+q+o(1/t)=x
\end{equation}
or, equivalently, $t^b \eee^{o(1/t)}=\eee^{x-q}$. As a consequence, $t(x)=\exp((x-q)/b)(1+\varepsilon(x))$, where $\varepsilon$ satisfies $\lim_{x\to\infty}\varepsilon(x)=0$. Plugging this into \eqref{eq:original} we infer $\varepsilon(x)=o(\eee^{-x/b})$ as $x\to\infty$ and thereupon \eqref{eq:expansion}.

Let now $\alpha\in (1/2,1)$. Starting with
\begin{equation}\label{eq:original2}
r_\alpha t^{-1+1/\alpha}+b\gamma_\alpha 
+o(1/t)=x
\end{equation}
we conclude that $t(x)=(r_\alpha^{-1}(x-b\gamma_\alpha 
)+\delta(x))^{\alpha/(1-\alpha)}$ for some $\delta$ satisfying $\lim_{x\to\infty}\delta(x)=0$. Plugging this expression into \eqref{eq:original2} we obtain $\delta(x)=o(x^{-\alpha/(1-\alpha)})$ as $x\to\infty$. Thus, representation \eqref{eq:expansion2} does indeed hold.
\end{proof}

\begin{proof}[Proof of Proposition \ref{prop:asymp2}]
The proof of Proposition \ref{prop:asymp1} goes through with the exception of a few places that we now point out. The major distinction in the present setting is that $-L^\prime(t)\sim r/t$ as $t\to\infty$ rather than $L^\prime(t)=o(1/t)$. By Lemma \ref{lem:aux2}, in formula \eqref{eq:inter4} $L^\prime(t)=-rt^{-1}+o(1/t)$ and $L^\prime(t)+tL^{\prime\prime}(t)=o(1/t)$ as $t\to\infty$. Also, $R(t)=o(1/t)$ as $t\to\infty$. Indeed,
\begin{multline*}
\int_1^\infty |f_t^{\prime\prime}(x)|{\rm d}x=t^{-(1+1/\alpha)}\int_0^t\big|(\alpha+1)y^{1/\alpha}L^\prime(y)+(3\alpha+1) y^{1+1/\alpha}L^{\prime\prime}(y)+\alpha y^{2+1/\alpha}L^{\prime\prime\prime}(y)\big|{\rm d}y\\=o(1/t),\quad t\to\infty
\end{multline*}
because the integrand is $o(y^{-1+1/\alpha})$ by Lemma \ref{lem:aux2}. Hence, formula \eqref{eq:bt} transforms into
\begin{equation*}
B_\alpha(t)=\int_1^{\lfloor t^\beta\rfloor}x^{-\alpha}L^\prime(t/x^\alpha){\rm d}x-rt^{-1}/2+o(1/t),\quad t\to\infty.
\end{equation*}
By Lemma \ref{lem:aux2}(b), in formula \eqref{eq:inter5}
$$-\alpha^{-1}t^{-1+1/\alpha}\int_t^\infty y^{-1/\alpha}L^\prime(y){\rm d}y=rt^{-1}+o(1/t),\quad t\to\infty.$$ Combining pieces together we conclude that, in comparison to the case $\mmp\{\eta=b\}\in (0,1)$ treated in Proposition \ref{prop:asymp1}, the asymptotic expansions of $\sum_{k\geq 1}k^{-\alpha}\psi^\prime(t/k^\alpha)$ have the additional summand $rt^{-1}/2$, that is, formulae \eqref{eq:formula 11} and \eqref{eq:formula 21} do indeed hold.

The argument leading to \eqref{eq:expansion11} and \eqref{eq:expansion22} is similar to that used to prove \eqref{eq:expansion} and \eqref{eq:expansion2}. For instance, to obtain \eqref{eq:expansion11} we represent the solution $t$ to
\begin{equation}\label{eq:repr2}
b\log t+q+r/(2t)+o(1/t)=x
\end{equation}
in the form $t(x)=\exp((x-q)/b)(1+\varepsilon(x))$ with $\lim_{x\to\infty}\varepsilon(x)=0$. Substituting this into \eqref{eq:repr2} we obtain \eqref{eq:expansion11}.

\end{proof}

\subsection{Proof of Propositions \ref{prop:first part1} and \ref{prop:first part2}}

We start by addressing the claim made in Remark \ref{rem:cont}.

\noindent {\sc Proof of \eqref{eq:cont}}. According to Lemma \ref{lem:aux1}(a), given $\varepsilon>0$ there exists $B>0$ such that $|\psi(x)-bx|<\varepsilon x$ whenever $x\geq B$. Using such a $B$ write $$(1-\alpha)\kappa_\alpha=(1-\alpha)\Big(\int_0^B x^{-1-1/\alpha}\psi(x){\rm d}x +\int_B^\infty x^{-1-1/\alpha}\psi(x){\rm d}x\Big).$$ The first relation of Lemma \ref{lem:at zero} ensures that $\int_0^Bx^{-2}\psi(x){\rm d}x<\infty$, whence $$\lim_{\alpha\to 1-0}(1-\alpha)\int_0^B x^{-1-1/\alpha}\psi(x){\rm d}x=0.$$ Further, $$(1-\alpha)\int_B^\infty x^{-1-1/\alpha}\psi(x){\rm d}x\leq (1-\alpha)(b+\varepsilon)\int_B^\infty x^{-1/\alpha}{\rm d}x\leq \alpha (b+\varepsilon)B^{1-1/\alpha}.$$ Thus, $\limsup_{\alpha\to 1-0}(1-\alpha)\kappa_\alpha\leq b$. The proof of the converse inequality for the lower limit is analogous.

\begin{proof}[Proof of Proposition \ref{prop:first part1}]
For a fixed $t>0$, put $f_t(x):=\psi(t/x^\alpha)-(t/x^\alpha)\psi^\prime(t/x^\alpha)$ for $x>0$. Then
\begin{equation}\label{eq:limit}
\lim_{x\to\infty}f_t(x)=\lim_{x\to\infty}f_t^\prime(x)=0.
\end{equation}
The former is a consequence of $\psi(0)=0$ and $\lim_{y\to 0}y\psi^\prime(y)=0$, see Lemma \ref{lem:at zero}. The latter follows from $f_t^\prime(x)=\alpha (t^2/x^{2\alpha+1})\psi^{\prime\prime}(t/x^\alpha)$ and the fact that $\lim_{y\to 0}\psi^{\prime\prime}(y)=\me [\eta^2]<\infty$ which holds by Lemma \ref{lem:at zero}.

In view of \eqref{eq:limit}, an application of formula \eqref{eq:EulMas infinite} with $f=f_t$ yields $$\sum_{j\geq 1}f_t(j)=\int_1^\infty f_t(x){\rm d}x+(f_t(1))/2-(f_t^\prime(1))/12+R_1(t),$$ where $|R_1(t)|\leq (1/12)\int_1^\infty |f_t^{\prime\prime}(x)|{\rm d}x$. By Lemma \ref{lem:at zero}(a,b), $f_t(1)=\psi(t)-t\psi^\prime(t)=\log \theta+o(1)$ as $t\to\infty$. Further, by Lemma \ref{lem:aux1}(c), $f_t^\prime(1)=\alpha t^2 \psi^{\prime\prime}(t)\to 0$ as $t\to\infty$. Now we intend to prove that
\begin{equation}\label{eq:rt}
\lim_{t\to\infty} R_1(t)=0.
\end{equation}
To this end, noting that
$$
f_t^{\prime\prime}(x)=-\alpha t^2 \big((2\alpha+1)x^{-2\alpha-2}\psi^{\prime\prime}(t/x^\alpha)+\alpha t x^{-3\alpha-2}\psi^{\prime\prime\prime}(t/x^\alpha)\big),
$$
we obtain $$\int_1^\infty (\alpha t/x^{\alpha+1})(t/x^\alpha)^{(\alpha+1)/\alpha}\psi^{\prime\prime}(t/x^\alpha){\rm d}x=t^{-1/\alpha}\int_0^t x^{1+1/\alpha}\psi^{\prime\prime}(x){\rm d}x~\to~0,\quad t\to\infty
$$
and
$$\int_1^\infty (\alpha t/x^{\alpha+1})(t/x^\alpha)^{(2\alpha+1)/\alpha}\big|\psi^{\prime\prime\prime}(t/x^\alpha)\big| {\rm d}x=t^{-1/\alpha}\int_0^t x^{2+1/\alpha}\big|\psi^{\prime\prime\prime}(x)\big|{\rm d}x~\to~0,\quad t\to\infty.
$$
Here, the limit relations are secured by $\lim_{y\to 0}y^2\psi^{\prime\prime}(y)=0$ and $\lim_{y\to 0}y^3|\psi^{\prime\prime\prime}(y)|=0$, respectively, see Lemma \ref{lem:aux1}(c). The proof of \eqref{eq:rt} is complete.

Write $f_t(x)=\big(\psi(t/x^\alpha)-(\alpha t/x^\alpha)\psi^\prime(t/x^\alpha)\big)-((1-\alpha) t/x^\alpha)\psi^\prime(t/x^\alpha)=:f_{t,1}(x)-f_{t,2}(x)$. Observe that $f_t=f_{t,1}$ in the case $\alpha=1$ and that $f_{t,1}(x)=(xg_t(x))^\prime$, where $g_t(x):=\psi(t/x^\alpha)$. This yields
\begin{equation}\label{eq:inter3}\int_1^\infty f_{t,1}(x){\rm d}x=\lim_{y\to\infty}(yg_t(y))-g_t(1)=-\psi(t)=-bt-\log\theta+o(1),\quad t\to\infty.
\end{equation}
Here, $\lim_{y\to\infty}yg_t(y)=0$ by the first formula in Lemma \ref{lem:at zero} and the last equality is ensured by Lemma \ref{lem:aux1}(a). Combining fragments together we arrive at \eqref{eq:first 1} in the case $\alpha=1$.

Assume now that $\alpha\in (1/2, 1)$. Changing the variable $y=x/t^\alpha$ and then integrating by parts we infer
\begin{multline*}
\int_1^\infty f_{t,2}(x){\rm d}x=\alpha^{-1}(1-\alpha)t^{1/\alpha}\int_0^t y^{-1/\alpha}\psi^\prime(y){\rm d}y\\=\alpha^{-1}(1-\alpha)t^{1/\alpha}\Big(t^{-1/\alpha}\psi(t)+\alpha^{-1}\int_0^\infty y^{-1-1/\alpha}\psi(y){\rm d}y-\alpha^{-1}\int_t^\infty y^{-1-1/\alpha}\psi(y){\rm d}y\Big).
\end{multline*}
In view of Lemma \ref{lem:aux1}(a) and the first relation in Lemma \ref{lem:at zero} the integral $\kappa_\alpha=\int_0^\infty y^{-1-1/\alpha}\psi(y){\rm d}y$ converges. Since $$t^{1/\alpha}\int_t^\infty y^{-1-1/\alpha}\psi(y){\rm d}y=
t^{1/\alpha}\int_t^\infty y^{-1-1/\alpha}(by+\log \theta+o(1)){\rm d}y=\alpha(1-\alpha)^{-1}bt+\alpha \log \theta+o(1)$$ as $t\to\infty$, we conclude that $$\int_1^\infty f_{t,2}(x){\rm d}x=\alpha^{-1}(1-\alpha)\psi(t)+\alpha^{-2}(1-\alpha)\kappa_\alpha t^{1/\alpha}-\alpha^{-1}bt-\alpha^{-1}(1-\alpha)\log \theta+o(1),\quad t\to\infty.$$ Recalling Lemma \ref{lem:aux1}(a), this in combination with \eqref{eq:inter3} proves
\begin{multline*}
\int_1^\infty f_t(x){\rm d}x=\int_1^\infty (f_{t,1}(x)-f_{t,2}(x)){\rm d}x\\=-\alpha^{-1}\psi(t)-\alpha^{-2}(1-\alpha)\kappa_\alpha t^{1/\alpha}+\alpha^{-1}bt+\alpha^{-1}(1-\alpha)\log \theta+o(1)\\=-\alpha^{-2}(1-\alpha)\kappa_\alpha t^{1/\alpha}-\log \theta+o(1),\quad t\to\infty.
\end{multline*}
The proof of Proposition \ref{prop:first part1} is complete.
\end{proof}

\begin{proof}[Proof of Proposition \ref{prop:first part2}]
Fix any $\beta>2/(2\alpha-1)$. By Lemma \ref{lem:at zero}, $\psi(t)\sim \me [\eta^2]t^2/2$ and $\psi^\prime(t)\sim \me [\eta^2]t$ as $t\to 0$. Using
$$-t\sum_{k\geq \lfloor t^\beta\rfloor+1}k^{-\alpha}\psi^\prime(t/k^\alpha)\leq \sum_{k\geq \lfloor t^\beta\rfloor+1}\big(\psi(t/k^\alpha)-(t/k^\alpha)\psi^\prime(t/k^\alpha)\big)\leq \sum_{k\geq \lfloor t^\beta\rfloor+1}\psi(t/k^\alpha)$$ we conclude that $$\sum_{k\geq \lfloor t^\beta\rfloor+1}\big(\psi(t/k^\alpha)-(t/k^\alpha)\psi^\prime(t/k^\alpha)\big)=o(1),\quad t\to\infty.$$ Indeed, $$\sum_{k\geq \lfloor t^\beta\rfloor+1}\psi(t/k^\alpha)=O\Big(t^2\sum_{k\geq \lfloor t^\beta\rfloor+1}k^{-2\alpha}\Big)=O(t^{2-\beta(2\alpha-1)})=o(1),\quad t\to\infty,$$ and an analogous asymptotic estimate holds true for the lower bound involving $\psi^\prime$.

Recalling that $f_t(x)=\psi(t/x^\alpha)-(t/x^\alpha)\psi^\prime(t/x^\alpha)$ and setting
\begin{align*}
h_t(x)
:&=f_t(x)+r\log (t/x^\alpha)\\
&=\psi(t/x^\alpha)-(t/x^\alpha)\psi^\prime(t/x^\alpha)+r\log (t/x^\alpha),
\end{align*}
note that $$h_t^\prime(x)=(\alpha t^2/x^{1+2\alpha})\psi^{\prime\prime}(t/x^\alpha)-r\alpha/x$$ and $$h_t^{\prime\prime}(x)=-\alpha(((2\alpha+1)t^2/x^{2\alpha+2})\psi^{\prime\prime}(t/x^\alpha)+(\alpha t^3/x^{3\alpha+2})\psi^{\prime\prime\prime}(t/x^\alpha)-r/x^2).$$ We shall use a representation
\begin{multline}\label{eq:decomp}
\sum_{k=1}^{\lfloor t^\beta\rfloor} \big(\psi(t/k^\alpha)-(t/k^\alpha)\psi^\prime(t/k^\alpha)\big)=\Big(\sum_{k=1}^{\lfloor t^\beta\rfloor}h_t(k)-r\int_1^{\lfloor t^\beta\rfloor}\log (t/x^\alpha){\rm d}x\Big)\\+r\Big(\int_1^{\lfloor t^\beta\rfloor}\log (t/x^\alpha){\rm d}x-\sum_{k=1}^{\lfloor t^\beta\rfloor}\log (t/k^\alpha)\Big)=:C_1(t)+C_2(t).
\end{multline}
By Stirling's formula,
\begin{align*}
C_2(t)
&=r\Big(-\log t+\alpha\Big(\sum_{k=1}^{\lfloor t^\beta\rfloor}\log k-\int_1^{\lfloor t^\beta\rfloor}\log x\,{\rm d}x\Big)\Big)\\
&=-r\log t+r\alpha\Big((\lfloor t^\beta\rfloor+1/2)\log (\lfloor t^\beta\rfloor)-\lfloor t^\beta\rfloor+(1/2)\log (2\pi)\Big)
+o(1)\\
&\hspace{1cm}-\lfloor t^\beta\rfloor \log (\lfloor t^\beta\rfloor)+\lfloor t^\beta\rfloor-1)\\
&=-r\log t+(r\alpha/2)\log (\lfloor t^\beta\rfloor)+r\alpha((1/2)\log (2\pi)-1)+o(1),\quad t\to\infty.
\end{align*}
Next, applying formula \eqref{eq:EulMas finite} with $f=h_t$, $m=1$ and $n=\lfloor t^\beta\rfloor$ and recalling that
$h_t(x)=f_t(x)+r\log(t/x^\alpha)$, we conclude that 
\begin{equation*}
C_1(t)=\int_1^{\lfloor t^\beta\rfloor}f_t(x){\rm d}x+(h_t(1)+h_t(\lfloor t^\beta\rfloor))/2+(h_t^\prime(\lfloor t^\beta\rfloor)-h_t^\prime(1))/12+R_2(t),
\end{equation*}
where $R_2(t)\leq (1/12)\int_1^{\lfloor t^\beta \rfloor}|h_t^{\prime\prime}(x)|{\rm d}x$. By Lemma \ref{lem:aux2}, as $t\to\infty$,
$$
h_t(1)=\psi(t)-t\psi^\prime(t)+r\log t
=r+\log\left(\lambda\Gamma(r+1)\right)+o(1)
$$
and
$$
h^\prime_t(1)=\alpha(t^2\psi^{\prime\prime}(t)-r)=o(1).
$$
Further, 
by Lemma \ref{lem:at zero}, as $t\to\infty$,
\begin{align*}
h_t(\lfloor t^\beta\rfloor)&=\psi(t/(\lfloor t^\beta\rfloor)^\alpha)-(t/(\lfloor t^\beta\rfloor)^\alpha)\psi^\prime(t/(\lfloor t^\beta\rfloor)^\alpha)+r(\log t-\alpha \log (\lfloor t^\beta\rfloor))\\
&=r(\log t-\alpha \log (\lfloor t^\beta\rfloor))+o(1)
\end{align*}
and $$h^\prime_t(\lfloor t^\beta\rfloor)=(\alpha t^2/(\lfloor t^\beta\rfloor)^{1+2\alpha})\psi^{\prime\prime}(t/(\lfloor t^\beta\rfloor)^\alpha)-r\alpha/\lfloor t^\beta\rfloor=o(1).$$ The relation $R_2(t)=o(1)$ as $t\to\infty$ follows from
\begin{multline*}
\int_1^\infty |h_t^{\prime\prime}(x)|{\rm d}x=t^{-1/\alpha}\int_1^\infty\big|(2\alpha+1)(t/x^\alpha)^{1+1/\alpha}(-\alpha t/x^{\alpha+1})\psi^{\prime\prime}(t/x^\alpha)\\+\alpha (t/x^\alpha)^{2+1/\alpha}(-\alpha t/x^{\alpha+1})\psi^{\prime\prime\prime}(t/x^\alpha)-r(t/x^\alpha)^{-1+1/\alpha}(-\alpha/x^{\alpha+1})\big|{\rm d}x\\=t^{-1/\alpha}\int_0^t \big|(2\alpha+1)y^{1+1/\alpha}\psi^{\prime\prime}(y)+\alpha y^{2+1/\alpha}\psi^{\prime\prime\prime}(y)-ry^{-1+1/\alpha}\big|{\rm d}y=o(1),\quad t\to\infty
\end{multline*}
because the integrand is $o(y^{-1+1/\alpha})$ by Lemma \ref{lem:aux2}.

Assume that $\alpha=1$. Since $f_t(x)=(xg_t(x))^\prime$, where $g_t(x)=\psi(t/x)$, we obtain
\begin{align*}
\int_1^{\lfloor t^\beta\rfloor}f_t(x){\rm d}x
&=\lfloor t^\beta\rfloor g_t(\lfloor t^\beta\rfloor)-g_t(1)
=\lfloor t^\beta\rfloor \psi(t/\lfloor t^\beta\rfloor)-\psi(t)\\
&=-bt+r\log t-\log(\lambda\Gamma(r+1))+o(1),\quad t\to\infty
\end{align*}
having utilized Lemma \ref{lem:aux2} and $\lfloor t^\beta\rfloor \psi(t/\lfloor t^\beta\rfloor)\sim \me [\eta^2]t^{2-\beta}/2=o(1)$ as $t\to\infty$, see Lemma \ref{lem:at zero}. Combining fragments together we obtain \eqref{eq:first 11}.

Assume now that $\alpha\in (1/2,1)$. Mimicking the argument used at the beginning of the proof for the sum one can show that $\int_{\lfloor t^\beta\rfloor}^\infty f_t(x){\rm d}x=o(1)$, whence
$$
\int_1^{\lfloor t^\beta\rfloor} f_t(x){\rm d}x=\int_1^\infty f_t(x){\rm d}x+o(1),\quad t\to\infty.
$$
From the proof of Proposition \ref{prop:first part1} we know that $$\int_1^\infty f_t(x){\rm d}x=-\alpha^{-1}\psi(t)-\alpha^{-2}(1-\alpha)\kappa_\alpha t^{1/\alpha}+\alpha^{-2}(1-\alpha)t^{1/\alpha}\int_t^\infty y^{-1-1/\alpha}\psi(y){\rm d}y.$$ By Lemma \ref{lem:aux2}(a), as $t\to\infty$,
\begin{multline*}
t^{1/\alpha}\int_t^\infty y^{-1-1/\alpha}\psi(y){\rm d}y=t^{1/\alpha}\int_t^\infty y^{-1-1/\alpha}(by-r\log y+
\log(\lambda \Gamma(r+1))+o(1)){\rm d}y\\=b\alpha(1-\alpha)^{-1}t-r\alpha \log t-r\alpha^2+\alpha\log(\lambda\Gamma(r+1))+o(1).
\end{multline*}
To calculate the integral involving $\log$ we have used the fact that $z\mapsto \alpha^{-2}\eee^{-z/\alpha}z$, $z>0$ is a density of the gamma distribution with parameters $1/\alpha$ and $2$. The corresponding distribution tail is $z\mapsto \eee^{-z/\alpha}(z/\alpha+1)$, $z\geq 0$. Thus, $$\int_t^\infty y^{-1-1/\alpha}\log y\,{\rm d}y=\int_{\log t}^\infty \eee^{-z/\alpha}z\,{\rm d}z=\alpha^2 t^{-1/\alpha}(\alpha^{-1}\log t+1).$$ This together with another application of Lemma \ref{lem:aux2} yields
\begin{align*}
&\int_1^{\lfloor t^\beta\rfloor}f_t(x){\rm d}x\\
&\hspace{1cm}=-b\alpha^{-1}t+r\alpha^{-1}\log t-
\alpha^{-1}\log(\lambda\Gamma(r+1))-\alpha^{-2}(1-\alpha)\kappa_\alpha t^{1/\alpha}+b\alpha^{-1}t\\
&\hspace{2cm}-r(1-\alpha)\alpha^{-1}\log t-r(1-\alpha)+\alpha^{-1}(1-\alpha)\log(\lambda\Gamma(r+1))+o(1)\\
&\hspace{1cm}=-\alpha^{-2}(1-\alpha)\kappa_\alpha t^{1/\alpha}+r\log t-r(1-\alpha)-\log(\lambda\Gamma(r+1))+o(1),\quad t\to\infty.
\end{align*}
Collecting pieces together we arrive at \eqref{eq:first 21}.
\end{proof}

\subsection{Proof of Proposition \ref{prop:second part}}

Here is a slight extension of formula \eqref{eq:change2}: for any bounded measurable $g:\mr\to \mathbb{C}$
\begin{equation}\label{eq:change}
\me^{(t)}[g(\eta_k)]=\frac{\me [\eee^{tS(\alpha)}g(\eta_k)]}{\me [\eee^{tS(\alpha)}]}=\frac{\me [\eee^{t\eta_k/k^\alpha}g(\eta_k)]\prod_{j\neq k}\me [\eee^{t\eta_j/j^\alpha}]}{\me [\eee^{t\eta_k/k^\alpha}]\prod_{j\neq k}\me [\eee^{t\eta_j/j^\alpha}]}=\frac{\me [\eee^{t\eta_k/k^\alpha}g(\eta_k)]}{\me [\eee^{t\eta_k/k^\alpha}]},
\end{equation}
where the second equality is justified by independence of $\eta_1$, $\eta_2,\ldots$

\begin{lemma}\label{lem:convchar}
Under the assumptions of either Theorem \ref{thm:main1} or Theorem \ref{thm:main2}, 
\begin{equation}\label{eq:conv}
\lim_{t\to\infty}\me^{(t)}[\eee^{{\rm i}ut^{1-1/(2\alpha)}S^{(t)}_0(\alpha)}]=\eee^{-\sigma^2_\alpha u^2/2},\quad u\in\mr
\end{equation}
with $\sigma^2_\alpha$ as defined in \eqref{eq:sigma}.
\end{lemma}
\begin{proof}
It is enough to prove that, for each $u\in\mr$,
\begin{equation}\label{eq:conv4}
\lim_{t\to\infty}\me^{(t)}[\eee^{ut^{1-1/(2\alpha)}S^{(t)}_0(\alpha)}]=\eee^{\sigma^2_\alpha u^2/2}.
\end{equation}
Indeed, \eqref{eq:conv4} ensures that the $\mmp^{(t)}$-distributions of the variables $t^{1-1/(2\alpha)}S_0^{(t)}(\alpha)$ converge weakly as $t\to\infty$ to the centered normal distribution with variance $\sigma^2_\alpha$. Relation \eqref{eq:conv} then follows by the L\'{e}vy continuity theorem for characteristic functions.

Under $\mmp^{(t)}$, the random variables $\eta_1$, $\eta_2,\ldots$ are still independent but not identically distributed. The former follows from \eqref{eq:change2}, and the latter follows from \eqref{eq:change}. Anyway, under $\mmp^{(t)}$, the variable $S_0(\alpha)=S(\alpha)-\me^{(t)}S(\alpha)$ is an infinite sum of independent centered random variables with finite second moments. Using this in combination with \eqref{eq:change} yields  
\begin{multline*}
\me^{(t)}[\eee^{ut^{1-1/(2\alpha)}S_0(\alpha)}]=\prod_{k\geq 1}\me^{(t)}[\exp((ut^{1-1/(2\alpha)}/k^\alpha)(\eta_k-\me^{(t)}[\eta_k]))]\\=\prod_{k\geq 1}\frac{\me\exp(((t+ut^{1-1/(2\alpha)})/k^\alpha)\eta-(ut^{1-1/(2\alpha)}/k^\alpha)\psi^\prime(t/k^\alpha)}{\me [\eee^{t\eta/k^\alpha}]}\\=\exp\Big(\sum_{k\geq 1}\big(\psi((t+ut^{1-1/(2\alpha)})/k^\alpha)-\psi(t/k^\alpha)-(ut^{1-1/(2\alpha)}/k^\alpha)\psi^\prime(t/k^\alpha)\big)\Big).
\end{multline*}
By the mean value theorem for twice differentiable function we further obtain, for some $\theta_k=\theta_k(t,u)\in (0,1)$,
\begin{multline*}
\me^{(t)}[\eee^{ut^{1-1/(2\alpha)}S_0(\alpha)}]=\exp\Big((t^{2-1/\alpha}u^2/2)\sum_{k\geq 1}k^{-2\alpha}\psi^{\prime\prime}((t+\theta_k ut^{1-1/(2\alpha)})/k^\alpha)\Big)\\=\exp\Big((t^{2-1/\alpha}u^2/2)\sum_{k\geq 1}k^{-2\alpha}\psi^{\prime\prime}(t/k^\alpha)\Big)\\\times \exp\Big((t^{2-1/\alpha}u^2/2)\sum_{k\geq 1}k^{-2\alpha}(\psi^{\prime\prime}((t+\theta_k ut^{1-1/(2\alpha)})/k^\alpha)-\psi^{\prime\prime}(t/k^\alpha))\Big).
\end{multline*}
By the mean value theorem for differentiable functions, for some $\vartheta_k=\vartheta_k(t,u)\in (0,1)$,
$$\sum_{k\geq 1}k^{-2\alpha}(\psi^{\prime\prime}((t+\theta_k ut^{1-1/(2\alpha)})/k^\alpha)-\psi^{\prime\prime}(t/k^\alpha))=ut^{1-1/(2\alpha)}\sum_{k\geq 1}k^{-3\alpha}\psi^{\prime \prime\prime}((t+\vartheta_k ut^{1-1/(2\alpha)})/k^\alpha).$$ Thus, \eqref{eq:conv4} follows if we can prove that 
\begin{equation}\label{eq:variance}
\lim_{t\to\infty} t^{2-1/\alpha}\sum_{k\geq 1}k^{-2\alpha}\psi^{\prime\prime}(t/k^\alpha)=\sigma_\alpha^2\in (0,\infty)
\end{equation}
and that, for each fixed $u\in\mr$,  
\begin{equation}\label{eq:lindeberg}
\lim_{t\to\infty} t^{3-3/(2\alpha)}\sum_{k\geq 1}k^{-3\alpha}|\psi^{\prime\prime\prime}((t+\vartheta_kut^{1-1/(2\alpha)})/k^\alpha)|=0.
\end{equation}

\noindent {\sc Proof of \eqref{eq:variance}}. 
We intend to show that the function $h$ defined by $h(x):=x^{-2\alpha}\psi^{\prime\prime}(x^{-\alpha})$ is directly Riemann integrable (dRi) on $[0,\infty)$. The function $\psi^{\prime\prime}$ is continuous on $[0,\infty)$. In view of Lemma \ref{lem:aux1}(c), for some $C_1>C_2>0$, $\psi^{\prime\prime}(x^{-\alpha})\leq C_1x^{2\alpha}$ for $x\in (0,1]$ and $\psi^{\prime\prime}(x^{-\alpha})\leq C_2$ for $x>0$. Hence, the function $h$ is continuous and bounded on $[0,\infty)$ and $0\leq h(x)\leq h_1(x)$ for $x>0$, where $h_1(x):=C_1\1_{[0,1]}(x)+C_2x^{-2\alpha}\1_{(1,\infty)}(x)$ for $x\geq 0$. Being a Lebesgue integrable nonincreasing function on $[0,\infty)$, $h_1$ is dRi on $[0,\infty)$, hence, so is $h$.

As a consequence, $$\lim_{t\to\infty}t^{-1/\alpha}\sum_{k\geq 1}t^2k^{-2\alpha}\psi^{\prime\prime}(tk^{-\alpha})=\int_0^\infty x^{-2\alpha}\psi^{\prime\prime}(x^{-\alpha}){\rm d}x=\alpha^{-1}\int_0^\infty x^{1-1/\alpha}\psi^{\prime\prime}(x){\rm d}x=\sigma_\alpha^2\in (0,\infty),$$ which proves \eqref{eq:variance}. Observe that $\sigma_1^2=\int_0^\infty \psi^{\prime\prime}(x){\rm d}x=\lim_{s\to\infty}\psi^\prime(s)-\lim_{s\to 0}\psi^\prime(x)=b$.

\noindent {\sc Proof of \eqref{eq:lindeberg}}. Fix any $u\geq 0$. The proof for fixed $u<0$ is analogous. By Lemmas \ref{lem:aux1}(c) and \ref{lem:aux2}(c), for large enough $t_1>0$ there exists $c_1>0$ such that $|\psi^{\prime\prime\prime}(t)|\leq c_1 t^{-3}$ whenever $t\geq t_1$. Hence, for positive integer $k\leq (t/t_1)^{1/\alpha}\leq ((t+\vartheta_kut^{1-1/(2\alpha)})/t_1)^{1/\alpha}$ $$|\psi^{\prime\prime\prime}((t+\vartheta_k ut^{1-1/(2\alpha)})/k^\alpha)|\leq c_1k^{3\alpha}/(t+\vartheta_kut^{1-1/(2\alpha)})^3 \leq c_1k^{3\alpha}t^{-3}.$$ This entails $$t^{3-3/(2\alpha)}\sum_{k=1}^{(t/t_1)^{1/\alpha}} k^{-3\alpha}|\psi^{\prime\prime\prime}((t+\vartheta_kut^{1-1/(2\alpha)})/k^\alpha)|\leq c_1 t^{-3/(2\alpha)}(t/t_1)^{1/\alpha}~\to~0,\quad t\to\infty.$$ Put $t_2:=2t_1$. By Lemma \ref{lem:at zero}, there exists $c_2>0$ such that $|\psi^{\prime\prime\prime}(t)|\leq c_2 t^{-1}$ whenever $t\in (0, t_2]$. For large enough $t$ and any $k\in\mn$, $t+\vartheta_kut^{1-1/(2\alpha)}\leq 2t$. Hence, for such $t$ and $k\geq (2t/t_2)^{1/\alpha}\geq ((t+\vartheta_kut^{1-1/(2\alpha)})/t_2)^{1/\alpha}$ $$|\psi^{\prime\prime\prime}((t+\vartheta_kut^{1-1/(2\alpha)})/k^\alpha)|\leq c_2k^\alpha /(t+\vartheta_kut^{1-1/(2\alpha)}) \leq c_2k^{\alpha}t^{-1}.$$ As a consequence,
\begin{multline*}
t^{3-3/(2\alpha)}\sum_{k\geq (t/t_1)^{1/\alpha}} k^{-3\alpha}|\psi^{\prime\prime\prime}((t+\vartheta_kut^{1-1/(2\alpha)})/k^\alpha)|\leq c_2 t^{2-3/(2\alpha)}\sum_{k\geq (t/t_1)^{1/\alpha}}k^{-2\alpha}\\=t^{2-3/(2\alpha)}O(t^{1/\alpha-2})~\to~0,\quad t\to\infty.
\end{multline*}
This finishes the proof of \eqref{eq:lindeberg}.

The proof of Lemma \ref{lem:convchar} is complete. 
\end{proof}

We need a local limit theorem which resembles Stone's local limit theorem for standard random walks attracted to a Brownian motion. While Stone's theorem deals with intervals of fixed length $h$, we treat intervals of length $ht^{-1/2}$ as $t\to\infty$.
\begin{thm}\label{thm:Stone}
Assume that $\eta\leq b$ a.s., $\me[\eta]=0$ and
$\me[\eta^2] 
\in (0,\infty)$. Then, for each $h>0$, $$\lim_{t\to\infty}\sup_{x\in\mr} \big|t^{1/(2\alpha)}\mmp^{(t)}\{S_0(\alpha)\in (xt^{-1+1/(2\alpha)}, xt^{-1+1/(2\alpha)}+ht^{-1}]\}-hn_\alpha(x)|=0,$$ where $n_\alpha(x):=(2\pi\sigma_\alpha^2)^{-1/2}\eee^{-x^2/(2\sigma_\alpha^2)}$ for $x\in\mr$.
\end{thm}

We first prove an auxiliary result.
\begin{lemma}\label{lem:bound at zero}
Under the assumptions of Theorem \ref{thm:Stone}, for all $t>0$, there exist positive constants $c$ and $\rho$ such that
$$\sup_{k\geq t^{1/\alpha}}\big|\me^{(t)}[\eee^{{\rm i}u\eta_k}]\big|\leq \eee^{-cu^2}$$ for all $u\in\mr$ satisfying $|u|\leq \rho$.
\end{lemma}
\begin{proof}
Fix any $u\in\mr$. Using formula \eqref{eq:change} with $g(x)=\eee^{{\rm i}ux}$ we infer $$\big|\me^{(t)}[\eee^{{\rm i}u\eta_k}]\big|=\frac{\big|\me [\eee^{({\rm i}u+t/k^\alpha)\eta_k}]\big|}{\me [\eee^{t\eta_k/k^\alpha}]}$$ and thereupon $$\sup_{k\geq t^{1/\alpha}}\big|\me^{(t)}[\eee^{{\rm i}u\eta_k}]\big|\leq \sup_{s\in [0,1]}\frac{\big|\me [\eee^{({\rm i}u+s)\eta}]\big|}{\me [\eee^{s\eta}]}.$$

Write $$\me [\eee^{({\rm i}u+s)\eta}]=\me [\eee^{s\eta}]+{\rm i}u\me [\eta \eee^{s\eta}]-2^{-1}u^2\me [\eta^2\eee^{s\eta}]+\me [\eee^{s\eta}(\eee^{{\rm i}u\eta}-1-{\rm i}u\eta+2^{-1}u^2\eta^2)].$$ By Lemma 3.3.7 on p.~115 in \cite{Durrett:2010},
\begin{equation}\label{eq:durr}
\big|\eee^{{\rm i}x}-1-{\rm i}x+ 2^{-1}x^2\big|\leq \min(|x|^3/6, |x|^2),\quad x\in\mr.
\end{equation}
As a consequence, 
for any $\kappa\in (0,1)$,
\begin{multline*}
\big|\me [\eee^{s\eta}(\eee^{{\rm i}u\eta}-1-{\rm i}u\eta+2^{-1}u^2\eta^2)]\big|\leq 6^{-1}|u|^3\me [|\eta|^3\eee^{s\eta}\1_{\{|\eta|\leq |u|^{-\kappa}\}}]+u^2 \me[\eta^2\eee^{s\eta}\1_{\{|\eta|>|u|^{-\kappa}\}}]\\\leq 6^{-1}|u|^{3-\kappa}\me [\eta^2\eee^{s\eta}]+u^2 \me[\eta^2\eee^{s\eta}\1_{\{|\eta|>|u|^{-\kappa}\}}].
\end{multline*}
Observe that $1/A:=\inf_{s\in [0,1]}\me [\eee^{s\eta}]\in (0,\infty)$ and that $\eee^{s\eta}\leq \eee^{sb}$ a.s. Hence,
\begin{equation*}
\sup_{s\in [0,1]}\frac{\big|\me [\eee^{s\eta}(\eee^{{\rm i}u\eta}-1-{\rm i}u\eta+2^{-1}u^2\eta^2)]\big|}{\me [\eee^{s\eta}]}\leq A\eee^b (6^{-1}|u|^{3-\kappa}\me [\eta^2]+u^2 \me[\eta^2\1_{\{|\eta|>|u|^{-\kappa}\}}].
\end{equation*}
This proves
\begin{equation}\label{eq:inter1}
\lim_{u\to 0} u^{-2}\sup_{s\in [0,1]}\,\frac{\big|\me [\eee^{s\eta}(\eee^{{\rm i}u\eta}-1-{\rm i}u\eta+2^{-1}u^2\eta^2)]\big|}{\me [\eee^{s\eta}]}=0.
\end{equation}
Further, with $\varphi(s):=\me [\eee^{s\eta}]$,
\begin{multline*}
\frac{\me [\eee^{s\eta}]+{\rm i}u\me [\eta \eee^{s\eta}]-2^{-1}u^2\me [\eta^2\eee^{s\eta}]}{\me [\eee^{s\eta}]}\eee^{-{\rm i}u \psi^\prime(s)}\\=\Big(1+{\rm i}u\psi^\prime(s)-2^{-1}u^2\frac{\varphi^{\prime\prime}(s)}{\varphi(s)}\Big)\big(1-{\rm i}u\psi^\prime(s)-2^{-1}u^2 (\psi^\prime(s))^2\big)\\+\Big(1+{\rm i}u\psi^\prime(s)-2^{-1}u^2\frac{\varphi^{\prime\prime}(s)}{\varphi(s)}\Big)(\eee^{-{\rm i}u\psi^\prime(s)}-1+{\rm i}u\psi^\prime(s)+2^{-1}u^2 (\psi^\prime(s))^2)=:K_1(s,u)+K_2(s,u).
\end{multline*}
Since $\me[\eta]=0$, the function $\psi$ is nondecreasing on $[0,\infty)$, whence $\psi^\prime(s)\geq 0$ for $s\geq 0$. In view of \eqref{eq:durr}, $$\big|\eee^{-{\rm i}u\psi^\prime(s)}-1+{\rm i}u\psi^\prime(s)+2^{-1}u^2 (\psi^\prime(s))^2\big|\leq 6^{-1}|u|^3 (\psi^\prime(s))^3$$ and thereupon 
$$\lim_{u\to 0}u^{-2}\sup_{s\in [0,1]}|K_2(s,u)|=0.$$ Observe that $$\frac{\varphi^{\prime\prime}(s)}{\varphi(s)}-(\psi^\prime(s))^2=\psi^{\prime\prime}(s)$$ and that $\psi^{\prime\prime}(s)\geq 0$ by convexity of $\psi$. With this at hand, we conclude that $$K_1(s,u)=1-2^{-1}u^2 \psi^{\prime\prime}(s)+2^{-1}{\rm i}u^3\psi^\prime(s)\psi^{\prime\prime}(s)+4^{-1}u^4 (\psi^\prime(s))^2\frac{\varphi^{\prime\prime}(s)}{\varphi(s)}$$ and that $$\lim_{u\to 0}u^{-2}\sup_{s\in [0,1]}|K_1(s,u)-1+2^{-1}u^2\psi^{\prime\prime}(s)|=0.$$
The function $\psi$ is strictly log-convex on $[0,\infty)$ with $\psi^{\prime\prime}(0)=\me[\eta^2]$. This entails $1/B:=\inf_{s\in [0,1]}\psi^{\prime\prime}(s)\in (0,\infty)$.
Thus,
\begin{multline*}
\sup_{s\in [0,1]}\Big|\frac{\me [\eee^{s\eta}]+{\rm i}u\me [\eta \eee^{s\eta}]-2^{-1}u^2\me [\eta^2\eee^{s\eta}]}{\me [\eee^{s\eta}]}\Big|=\sup_{s\in [0,1]}\Big|\frac{\me [\eee^{s\eta]}+{\rm i}u\me [\eta \eee^{s\eta}]-2^{-1}u^2\me [\eta^2\eee^{s\eta}]}{\me [\eee^{s\eta}]}\eee^{-{\rm i}u \psi^\prime(s)}\Big|\\ \leq \sup_{s\in [0,1]}|K_1(s,u)-1+2^{-1}u^2\psi^{\prime\prime}(s)|+\sup_{s\in[0,1]}|1-2^{-1}u^2\psi^{\prime\prime}(s)|+\sup_{s\in [0,1]}|K_2(s,u)|\\\leq 1-2^{-1}u^2 \inf_{s\in [0,1]}\psi^{\prime\prime}(s)+o(u^2),\quad u\to 0.
\end{multline*}
This in combination with \eqref{eq:inter1} shows that
\begin{equation*}
\sup_{s\in [0,1]}\frac{\big|\me [\eee^{({\rm i}u+s)\eta}]\big|}{\me [\eee^{s\eta}]}\leq 1-(2B)^{-1}u^2+o(u^2),\quad u\to 0,
\end{equation*}
where the term $o(u^2)$ is uniform in $s\in [0,1]$. In particular, there exists $u_0>0$ such that $|o(u^2)|\leq (4B)^{-1}u^2$ whenever $|u|\leq u_0$. Hence,
\begin{equation*}
\sup_{s\in [0,1]}\frac{\big|\me [\eee^{({\rm i}u+s)\eta}]\big|}{\me [\eee^{s\eta}]}\leq 1-(4B)^{-1}u^2\leq \eee^{-u^2/(4B)}
\end{equation*}
whenever $|u|\leq \min (u_0, (2B)^{1/2})$.
\end{proof}

We are ready to prove Theorem \ref{thm:Stone}.

\begin{proof}[Proof of Theorem \ref{thm:Stone}]
For $x\in\mr$, $\delta>0$ and $t>0$, put $$v_{t,\,\delta}(x):=\delta^{-1}\mmp^{(t)}\{S_0(\alpha)\in (x, x+\delta]\}.$$ Observe that $v_{t,\,\delta}$ is a 
density of the $\mmp^{(t)}$-distribution of $S_0(\alpha)-U_\delta$, where $S_0(\alpha)$ and $U_\delta$ are $\mmp^{(t)}$-independent, and $U_\delta$ has a uniform distribution on $(0,\delta)$. The corresponding characteristic function is
\begin{equation}\label{eq:char}
\int_\mr \eee^{{\rm i}zx}v_{t,\,\delta}(x){\rm d}x=\me^{(t)}\big[\eee^{{\rm i}zS_0(\alpha)}\big]\frac{1-\eee^{-{\rm i}\delta z}}{{\rm i}\delta z},\quad z\in\mr.
\end{equation}
Later in the proof we shall show that this characteristic function is absolutely integrable on $\mr$. This entails
\begin{equation}\label{eq:basic}
v_{t,\,\delta}(x)=\frac{1}{2\pi}\int_\mr \eee^{-{\rm i}zx}\me^{(t)}\big[\eee^{{\rm i}zS_0(\alpha)}\big]\frac{1-\eee^{-{\rm i}\delta z}}{{\rm i}\delta z}{\rm d}z,\quad x\in\mr.
\end{equation}

Changing the variable $z=ut^{1-1/(2\alpha)}$ we obtain $$v_{t,\,\delta}(xt^{1/(2\alpha)-1})=\frac{t^{1-1/(2\alpha)}}{2\pi}\int_\mr \eee^{-{\rm i}ux}\me^{(t)}\big[\eee^{{\rm i}ut^{1-1/(2\alpha)}S_0(\alpha)}\big]\frac{1-\eee^{-{\rm i}\delta ut^{1-1/(2\alpha)}}}{{\rm i}\delta ut^{1-1/(2\alpha)}}{\rm d}u.$$ This yields
\begin{multline*}
t^{1/(2\alpha)}\mmp^{(t)}\{S_0(\alpha)\in (xt^{-1+1/(2\alpha)}, xt^{-1+1/(2\alpha)}+ht^{-1}]\}=ht^{-1+1/(2\alpha)}v_{t,\,ht^{-1}}(xt^{-1+1/(2\alpha)})\\=\frac{h}{2\pi}\int_\mr \eee^{-{\rm i}ux}\me^{(t)}\big[\eee^{{\rm i}ut^{1-1/(2\alpha)}S_0(\alpha)}\big]\frac{1-\eee^{-{\rm i}hut^{-1/(2\alpha)}}}{{\rm i}hut^{-1/(2\alpha)}}{\rm d}u.
\end{multline*}
Noting that $n_\alpha(x)=(2\pi)^{-1}\int_\mr \eee^{-{\rm i}xu}\eee^{-\sigma_\alpha^2 u^2/2}{\rm d}u$, it suffices to prove that, for any $A>0$,
\begin{equation}\label{eq:1}
\lim_{t\to\infty}\sup_{x\in\mr}\Big|\int_{-A}^A \eee^{-{\rm i}ux}\Big(\me^{(t)}\big[\eee^{{\rm i}ut^{1-1/(2\alpha)}S_0(\alpha)}\big]\frac{1-\eee^{-{\rm i}hut^{-1/(2\alpha)}}}{{\rm i}hut^{-1/(2\alpha)}}-\eee^{-\sigma_\alpha^2 u^2/2}\Big){\rm d}u\Big|=0
\end{equation}
and that
\begin{equation}\label{eq:2}
\lim_{A\to\infty}\limsup_{t\to\infty}\sup_{x\in\mr}\Big|\int_{|u|>A} \eee^{-{\rm i}ux}\me^{(t)}\big[\eee^{{\rm i}ut^{1-1/(2\alpha)}S_0(\alpha)}\big]\frac{1-\eee^{-{\rm i}hut^{-1/(2\alpha)}}}{{\rm i}hut^{-1/(2\alpha)}}{\rm d}u\Big|=0;
\end{equation}
\begin{equation}\label{eq:3}
\lim_{A\to\infty}\sup_{x\in\mr}\Big|\int_{|u|>A} \eee^{-{\rm i}ux} \eee^{-\sigma_\alpha^2 u^2/2}{\rm d}u\Big|=0.
\end{equation}

\noindent {\sc Proof of \eqref{eq:1}}. The supremum in \eqref{eq:1} does not exceed $$\int_{-A}^A \Big|\me^{(t)}\big[\eee^{{\rm i}ut^{1-1/(2\alpha)}S_0(\alpha)}\big]\frac{1-\eee^{-{\rm i}hut^{-1/(2\alpha)}}}{{\rm i}hut^{-1/(2\alpha)}}-\eee^{-\sigma_\alpha^2 u^2/2}\Big|{\rm d}u.$$
By Lemma \ref{lem:convchar}, $\lim_{t\to\infty} \me^{(t)}\big[\eee^{{\rm i}ut^{1-1/(2\alpha)}S_0(\alpha)}\big]=\eee^{-\sigma_\alpha^2 u^2/2}$ for $u\in\mr$. The characteristic functions $u\mapsto (1-\eee^{-{\rm i}hut^{-1/(2\alpha)}})/({\rm i}hut^{-1/(2\alpha)})$ converge as $t\to\infty$ to $r$ the characteristic function of degenerate at $0$ distribution ($r(u)=1$ for $u\in\mr$). Since the convergence of characteristic functions is locally uniform, the latter integral converges to $0$ as $t\to\infty$.

\noindent {\sc Proof of \eqref{eq:2}}. Let $\rho$ be as in Lemma \ref{lem:bound at zero}. Write, for large $t$,
\begin{multline*}
\sup_{x\in\mr}\Big|\int_{|u|>A} \eee^{-{\rm i}ux}\me^{(t)}\big[\eee^{{\rm i}ut^{1-1/(2\alpha)}S_0(\alpha)}\big]\frac{1-\eee^{-{\rm i}hut^{-1/(2\alpha)}}}{{\rm i}hut^{-1/(2\alpha)}}{\rm d}u\Big|\\\leq \int_{A<|u|\leq \rho t^{1/(2\alpha)}}\big|\me^{(t)}\big[\eee^{{\rm i}ut^{1-1/(2\alpha)}S_0(\alpha)}\big]\big|{\rm d}u+\int_{|u|>\rho t^{1/(2\alpha)}}\big|\me^{(t)}\big[\eee^{{\rm i}ut^{1-1/(2\alpha)}S_0(\alpha)}\big]\big|{\rm d}u\\=:I_1(t,A)+I_2(t).
\end{multline*}
Observe that $$\big|\me^{(t)}\big[\eee^{{\rm i}ut^{1-1/(2\alpha)}S_0(\alpha)}\big]\big|=\prod_{k\geq 1}\big|\me^{(t)}\big[\eee^{{\rm i}ut^{1-1/(2\alpha)}k^{-\alpha}(\eta_k-\me^{(t)}\eta_k)}\big]\big|=\prod_{k\geq 1}\big|\me^{(t)}\big[\eee^{{\rm i}ut^{1-1/(2\alpha)}k^{-\alpha}\eta_k}\big]\big|.$$

\noindent {\sc Analysis of $I_1$}. Put $c_2:=\inf_{t\geq 1} t^{2-1/\alpha}\sum_{k\geq t^{1/\alpha}}k^{-2\alpha}$ and note that $c_2>0$ because
$$\lim_{t\to\infty} t^{2-1/\alpha}\sum_{k\geq t^{1/\alpha}}k^{-2\alpha}=(2\alpha-1)^{-1}.$$ If $|u|\leq \rho t^{1/(2\alpha)}$, then $|u|t^{1-1/(2\alpha)}/k^\alpha\leq \rho$ for all $k\geq t^{1/\alpha}$ and, according to Lemma \ref{lem:bound at zero}, for $t\geq 1$, $$\big|\me^{(t)}\big[\eee^{{\rm i}ut^{1-1/(2\alpha)}S_0(\alpha)}\big]\big|\leq \prod_{k\geq t^{1/\alpha}}\big|\me^{(t)}\big[\eee^{{\rm i}ut^{1-1/(2\alpha)}k^{-\alpha}\eta_k}\big]\big|\leq \exp(-cu^2 t^{2-1/\alpha}\sum_{k\geq t^{1/\alpha}}k^{-2\alpha})\leq \eee^{-cc_2 u^2}.$$ Since the function $u\mapsto \eee^{-cc_2 u^2}$ is integrable on $\mr$ and, for large $t$, $I_1(t,A)\leq \int_{|u|>A}\eee^{-cc_2 u^2}{\rm d}u$, we conclude that
$\lim_{A\to\infty}\limsup_{t\to\infty}I_1(t,A)=0$.

\noindent {\sc Analysis of $I_2$}. If $|u|>\rho t^{1/(2\alpha)}$, then $|u|t^{1-1/(2\alpha)}/k^\alpha\leq \rho$ whenever $k^\alpha\geq \rho^{-1}|u|t^{1-1/(2\alpha)}>t$. Invoking Lemma \ref{lem:bound at zero} once again we obtain, for $t\geq 1$,
\begin{multline*}
\big|\me^{(t)}\big[\eee^{{\rm i}ut^{1-1/(2\alpha)}S_0(\alpha)}\big]\big|\leq \prod_{k\geq (\rho^{-1}|u|t^{1-1/(2\alpha)})^{1/\alpha}}\big|\me^{(t)}\big[\eee^{{\rm i}ut^{1-1/(2\alpha)}k^{-\alpha}\eta_k}\big]\big|\\\leq \exp\Big(-cu^2 t^{2-1/\alpha}\sum_{k\geq (\rho^{-1}|u|t^{1-1/(2\alpha)})^{1/\alpha}}k^{-2\alpha})\leq \eee^{-cc_2\rho^{2-1/\alpha}|u|^{1/\alpha}t^{1/\alpha-1/(2\alpha^2)}}\leq \eee^{-cc_2\rho^{2-1/\alpha}|u|^{1/\alpha}}.
\end{multline*}
Thus, for $t\geq 1$,
$I_2(t)\leq \int_{|u|> \rho t^{1/(2\alpha)}}\eee^{-cc_2\rho^{2-1/\alpha}|u|^{1/\alpha}}{\rm d}u$ and thereupon $\lim_{t\to\infty} I_2(t)=0$.

The proof of relation \eqref{eq:3} is trivial, hence omitted. The proof of \eqref{eq:2} is complete. We note in passing that while dealing with \eqref{eq:2} we have shown that the characteristic function given in \eqref{eq:char} is absolutely integrable, thereby justifying \eqref{eq:basic}.

The proof of Theorem \ref{thm:Stone} is complete.
\end{proof}

\begin{proof}[Proof of Proposition \ref{prop:second part}]
Fix any $h\in (0,1)$ and write
\begin{multline*}
t^{1/(2\alpha)}\me^{(t)}\big[\eee^{-tS_0(\alpha)}\1_{\{S_0(\alpha)>0\}}\big]=t^{1/(2\alpha)}\int_{(0, \lfloor t^{1/(4\alpha)}\rfloor ht^{-1}]}\eee^{-tx}{\rm d}\mmp^{(t)}\{S_0(\alpha)\leq x\}\\+t^{1/(2\alpha)}\int_{(\lfloor t^{1/(4\alpha)}\rfloor ht^{-1}, \infty)}\eee^{-tx}{\rm d}\mmp^{(t)}\{S_0(\alpha)\leq x\}.
\end{multline*}
The second term is dominated by $t^{1/(2\alpha)}\eee^{-\lfloor t^{1/(4\alpha)}\rfloor h}$ and as such is negligible as $t\to\infty$. The first term is equal to
\begin{multline*}
t^{1/(2\alpha)}\sum_{k=0}^{\lfloor t^{1/(4\alpha)}\rfloor-1}\int_{(kh/t, (k+1)h/t]}\eee^{-tx}{\rm d}\mmp^{(t)}\{S_0\leq x\}\\\leq \sum_{k=0}^{\lfloor t^{1/(4\alpha)}\rfloor-1} \eee^{-kh}\big(t^{1/(2\alpha)}\mmp^{(t)}\{S_0(\alpha)\in (kh/t, (k+1)h/t]\}-hn_\alpha(kh t^{-1/(2\alpha)})\big)\\+h\sum_{k=0}^{\lfloor t^{1/(4\alpha)}\rfloor-1} \eee^{-kh}n_\alpha(kh t^{-1/(2\alpha)}).
\end{multline*}
By Theorem \ref{thm:Stone} with $x=kht^{-1/(2\alpha)}$, the first term on the right-hand side vanishes as $t\to\infty$. Given $\varepsilon>0$, $|n_\alpha(kh t^{-1/(2\alpha)})-n_\alpha(0)|\leq \varepsilon$ for all positive integers $k\leq \lfloor t^{1/(4\alpha)}\rfloor-1$ and large enough $t$. By a standard reasoning it follows that $\lim_{t\to\infty} \sum_{k=0}^{\lfloor t^{1/(4\alpha)}\rfloor-1} \eee^{-kh}(n_\alpha(kh t^{-1/(2\alpha)})-n_\alpha(0))=0$. Thus, we have proved that
$$\limsup_{t\to\infty} t^{1/(2\alpha)}\me^{(t)}\big[\eee^{-tS_0(\alpha)}\1_{\{S_0(\alpha)>0\}}\big]\leq n_\alpha(0) h(1-\eee^{-h})^{-1}.$$ Letting $h$ tend to $0+$ we infer $$\limsup_{t\to\infty} t^{1/(2\alpha)}\me^{(t)}\big[\eee^{-tS_0(\alpha)}\1_{\{S_0(\alpha)>0\}}\big]\leq n_\alpha(0)=(2\pi\sigma_\alpha^2)^{-1/2}.$$ The converse inequality for the limit inferior follows analogously. The proof of Proposition \ref{prop:second part} is complete.
\end{proof}

\section{Proofs of Theorems \ref{thm:dens1} and \ref{thm:dens2}}

As we have already mentioned in the introduction, the distribution of $S(\alpha)$ is absolutely continuous with a smooth density $g_\alpha$. 
For $t>0$, define $g_\alpha^{(t)}$ by
$$
g_\alpha^{(t)}(x)=\frac{\eee^{tx}g_\alpha(x)}{\me[\eee^{tS(\alpha)}]}, 
\quad x\in\mathbb{R}
$$
or equivalently 
\begin{equation}\label{eq:Cramer-inverse}
g_\alpha(x)=\me[\eee^{tS(\alpha)}]\eee^{-tx}g^{(t)}_\alpha(x),\quad x\in\mathbb{R}.
\end{equation}
The so defined $g_\alpha^{(t)}$ is a density of the the $\mmp^{(t)}$-distribution of $S(\alpha)$. It follows from the proof of Theorem \ref{thm:Stone} that the characteristic function $z\mapsto \me^{(t)}[\eee^{{\rm i}zS_0(\alpha)}]$, $x\in\mr$ is absolutely integrable. Hence, an application of the Fourier inversion formula yields
\begin{align*}
g_\alpha^{(t)}(\me^{(t)}[S(\alpha)])
=\frac{1}{2\pi}
\int_\mr 
\me^{(t)}[\eee^{{\rm i}zS_0(\alpha)}]{\rm d}z.
\end{align*}
Substituting $z=ut^{1-1/(2\alpha)}$, we then have
\begin{align*}
g_\alpha^{(t)}(\me^{(t)}[S(\alpha)])
=\frac{t^{1-1/(2\alpha)}}{2\pi}
\int_\mr 
\me^{(t)}[\eee^{{\rm i}ut^{1-1/(2\alpha)}S_0(\alpha)}]{\rm d}u.
\end{align*}
Combining this representation with the equality
$$
\frac{1}{\sqrt{2\pi\sigma_\alpha^2}}
=\frac{1}{2\pi}\int_\mr 
\eee^{-\sigma_\alpha^2u^2/2}{\rm d}u,
$$
we infer
$$
t^{1/(2\alpha)-1}g_\alpha^{(t)}(\me^{(t)}[S(\alpha)])
=\frac{1}{\sqrt{2\pi\sigma_\alpha^2}}
+\frac{1}{2\pi}\int_\mr 
\left(\me^{(t)}[\eee^{{\rm i}ut^{1-1/(2\alpha)}S_0(\alpha)}]
-\eee^{-\sigma_\alpha^2u^2/2}\right){\rm d}u.
$$
Arguing as in 
the proof of Theorem~\ref{thm:Stone}, we obtain 
$$
\lim_{t\to\infty}
\int_\mr 
\left(\me^{(t)}[\eee^{{\rm i}ut^{1-1/(2\alpha)}S_0(\alpha)}]
-\eee^{-\sigma_\alpha^2u^2/2}\right){\rm d}u=0
$$
and thereupon 
$$
g_\alpha^{(t)}(\me^{(t)}[S(\alpha)])~\sim~\frac{t^{1-1/(2\alpha)}}{\sqrt{2\pi\sigma_\alpha^2}},\quad t\to\infty.
$$
\bigskip
Plugging this into \eqref{eq:Cramer-inverse} yields
$$
g_\alpha(\me^{(t)}[S(\alpha)])~\sim~\me[\eee^{tS(\alpha)}]\eee^{-t\me^{(t)}[S(\alpha)]}\frac{t^{1-1/(2\alpha)}}{\sqrt{2\pi\sigma_\alpha^2}},\quad t\to\infty.
$$
Using the same $t=t(x)$ as before, that is, a unique solution to $\me^{(t)}[S(\alpha)]=x$, see \eqref{eq:t(x)},
we obtain
$$
g_\alpha(x)~\sim~\me[\eee^{t(x)S(\alpha)}]\eee^{-xt(x)}\frac{(t(x))^{1-1/(2\alpha)}}{\sqrt{2\pi\sigma_\alpha^2}},\quad x\to\infty.
$$
By Proposition \ref{prop:second part} and the argument given in Section \ref{sec:structure},
$$
\mmp\{S(\alpha)>x\}~\sim~\me[\eee^{t(x)S(\alpha)}]\eee^{-xt(x)} \frac{(t(x))^{-1/(2\alpha)}}{\sqrt{2\pi\sigma_\alpha^2}},\quad x\to\infty,
$$
whence $$g_\alpha(x)~\sim~ t(x)\mmp\{S(\alpha)>x\},\quad x\to\infty.$$ With this at hand, the desired result is now secured by 
the already known asymptotics of $\mmp\{S(\alpha)>x\}$ given in Theorems \ref{thm:main1} and \ref{thm:main2} and the asymptotic of $t(x)$ given in Propositions \ref{prop:asymp1} and \ref{prop:asymp2}.   

\noindent \textbf{Acknowledgment.} The authors thank Zakhar Kabluchko for very fruitful discussions on the topic of the paper and for a pointer to \cite{Rice:1973}.\\
A part of this work was done while A.I. was visiting Bielefeld in June 2023. Grateful acknowledgment is
made for financial support and hospitality.

\end{document}